\documentclass[a4paper,12pt,reqno]{amsart}

\usepackage{amssymb}
\usepackage{epsfig}
\usepackage{amsfonts}
\usepackage{amsmath}
\usepackage{euscript}
\usepackage{amscd}
\usepackage{amsthm}
\usepackage{enumitem}
\DeclareMathAlphabet{\mathpzc}{OT1}{pzc}{m}{it}
\usepackage{enumitem}
\usepackage{color}
\usepackage{mathtools}
\usepackage[pagebackref, hypertexnames=false, colorlinks, citecolor=red, linkcolor=blue, urlcolor=red]{hyperref}

\usepackage{marginnote}

\usepackage[normalem]{ulem}

\newcommand{\marginextend}[1]{ \addtolength{\oddsidemargin}{-#1}  \addtolength{\evensidemargin}{-#1}
  \addtolength{\textwidth}{#1}\addtolength{\textwidth}{#1}}
\newcommand{\updownextend}[1]{ \addtolength{\topmargin}{-#1}  \addtolength{\textheight}{#1}
\addtolength{\textheight}{#1}}
\marginextend{1.5cm}
\updownextend{0cm}
\allowdisplaybreaks[4]

\usepackage{pst-node}
\usepackage{tikz-cd}
\usepackage{mathrsfs}
\DeclareFontFamily{OT1}{pzc}{}
\DeclareFontShape{OT1}{pzc}{m}{it}{<-> s * [1.10] pzcmi7t}{}
\DeclareMathAlphabet{\mathpzc}{OT1}{pzc}{m}{it}

\DeclareSymbolFont{SY}{U}{psy}{m}{n}
\DeclareMathSymbol{\emptyset}{\mathord}{SY}{'306}

\theoremstyle{plain}

\newtheorem{thm}{Theorem}[section]
\newtheorem*{thm*}{Theorem}
\newtheorem{cor}[thm]{Corollary}
\newtheorem{lem}[thm]{Lemma}
\newtheorem{prop}[thm]{Proposition}
\newtheorem{defn}[thm]{Definition}
\newtheorem{rem}[thm]{Remark}
\newtheorem{ex}[thm]{Example}

\newtheorem{notation}[thm]{Notation}

\newtheoremstyle{named}{}{}{\itshape}{}{\bfseries}{.}{.5em}{#1 \thmnote{#3}}
\theoremstyle{named}

\numberwithin{equation}{section}

\def\C{{\mathbb C}}

\def\O{{\mathcal O}}

\def\norm#1{\left\|{#1}\right\|}

\def\l{\lambda}

\def\ra{\rightarrow}
\def\ov{\overline}

\def\m{\mathcal}
\def\mb{\mathbb}
\def\mr{\mathrm}

\def\beq{\begin{eqnarray}}
\def\eeq{\end{eqnarray}}
\def\beqa{\begin{eqnarray*}}
\def\eeqa{\end{eqnarray*}}

\def\del{\partial}

\def\ov{\overline}
\def\bl{\boldsymbol}

\def\sgn{{\rm sgn}}
\def\bl{\boldsymbol}
\def\deg{{\rm deg}}

\newcommand{\overbar}[1]{\mkern 1.5mu\overline{\mkern-1.5mu#1\mkern-1.5mu}\mkern 1.5mu}

\newcommand{\be}{\begin{equation}}
\newcommand{\ee}{\end{equation}}
\newcommand{\bea}{\begin{eqnarray}}
\newcommand{\eea}{\end{eqnarray}}
\newcommand{\Bea}{\begin{eqnarray*}}
\newcommand{\Eea}{\end{eqnarray*}}
\newcommand{\inner}[2]{\langle #1,#2 \rangle }%

\newcounter{cnt1}
\newcounter{cnt2}
\newcounter{cnt3}
\newcommand{\blr}{\begin{list}{$($\roman{cnt1}$)$}
 {\usecounter{cnt1} \setlength{\topsep}{0pt}
 \setlength{\itemsep}{0pt}}}
\newcommand{\bla}{\begin{list}{$($\alph{cnt2}$)$}
 {\usecounter{cnt2} \setlength{\topsep}{0pt}
 \setlength{\itemsep}{0pt}}}
\newcommand{\bln}{\begin{list}{$($\arabic{cnt3}$)$}
 {\usecounter{cnt3} \setlength{\topsep}{0pt}
 \setlength{\itemsep}{0pt}}}
\newcommand{\el}{\end{list}}

\DeclareMathOperator  {\Aut} {Aut}

\DeclareMathOperator  {\tr} {tr}

\title[Brown-Halmos type Theorems]{Brown-Halmos type Theorems on the proper images of bounded symmetric domains}

\author[Ghosh]{Gargi Ghosh$^\dag$}\address[Ghosh]{Silesian University in Opava, 746 01, Czech Republic} \address{Jagiellonian University, 30-348 Krakow, Poland}
\email[Ghosh]{gargighosh1811@gmail.com}
\author[Shyam Roy]{Subrata Shyam Roy}\address[Shyam Roy]{Indian Institute of Science Education And Research Kolkata, 741246, India}
\email[Shyam Roy]{ssroy@iiserkol.ac.in}

\thanks{$^\dag$This work is supported by postdoctoral fellowship from Silesian University in Opava under GA CR grant no. 21-27941S and is supported by the project No. 2022/45/P/ST1/01028 co-funded by the National Science Centre and the European Union Framework Programme for Research and Innovation Horizon 2020 under the Marie Sklodowska-Curie grant agreement No. 945339. For the purpose of Open Access, the author has applied a CC-BY public copyright licence to any Author Accepted Manuscript (AAM) version arising from this submission.}
\begin{document}
\begin{abstract}
    Let $\Omega\subseteq\mathbb C^n$ be a bounded symmetric domain  and $\bl f :\Omega \to \Omega^\prime\subseteq \mathbb C^n$ be a proper holomorphic mapping which is factored by a finite complex reflection group  $G.$ We identify a family of reproducing kernel Hilbert spaces on $\Omega^\prime$ arising naturally from the isotypic decomposition of the regular representation of $G$ on the Hardy space $H^2(\Omega).$ Each element of this family can be realized as a closed subspace of some $L^2$-space on the \v{S}ilov boundary of $\Omega^\prime$. The reproducing kernel Hilbert space associated to the sign representation of $G$ is the Hardy space $H^2(\Omega^\prime).$  We establish a Brown-Halmos type characterization for the Toeplitz operators on $H^2(\Omega^\prime),$ where $\Omega^\prime$ is the image of the open unit polydisc $\mathbb D^n$ in $\mathbb C^n$ under a proper holomorphic mapping factored by the finite complex reflection group $G(m,p,n).$ Moreover, we prove various multiplicative properties of Toeplitz operators on $H^2(\Omega^\prime)$,  where $\Omega^\prime$  is a proper holomorphic image of a bounded symmetric domain. 
\end{abstract}
\maketitle
\section{Introduction}
Let $\mb D$ denote the open unit disc in the complex plane $\mb C$ and $H^2(\mb D)$ denote the Hardy space on $\mb D$.  The study of Toeplitz operators on $H^2(\mb D)$ gained significant attention after the influential paper by Brown and Halmos \cite{BH1964} which explored the algebraic properties of these operators. A key result from their work provides a characterization of Toeplitz operators on $H^2(\mb D).$ This result was subsequently extended to $H^2(\mb D^n)$ in \cite{MSS18} which states that a bounded linear operator $T$ on $H^2(\mb D^n)$ is a Toeplitz operator if and only if $$T_j^*TT_j=T \text{ for every }j=1,\ldots,n,$$ where $T_j$ denotes the $j$-th coordinate multiplication operator on $H^2(\mb D^n).$ In this article, our primary objective is to establish a similar characterization for Toeplitz operators on $H^2(D),$ $D$ being a proper holomorphic image of the polydisc $\mb D^n.$ Inspired by \cite{BH1964}, we also prove various multiplicative properties of Toeplitz operators on $H^2(D)$ and in this case $D$  is allowed to be any proper holomorphic image of a bounded symmetric domain. To achieve this, we proceed as follows. 
\begin{itemize}
    \item[(i)] Firstly, we identify an appropriate notion of Hardy space $H^2(D)$ from a naturally occurring family of reproducing kernel Hilbert spaces.
    \item[(ii)] Subsequently, we show that $H^2(D)$ can be realized as a closed subspace of an $L^2$-space on the \v{S}ilov boundary of $D$, which leads to the study of Toeplitz operators on $H^2(D)$.
\end{itemize} %

We begin by recalling some known facts to lay the groundwork for our results. Let $\Omega$ be a domain in $\mb C^n$ and ${\rm Aut}(\Omega)$ be the group of all biholomorphic automorphisms of $\Omega.$ \begin{defn}\cite[p. 8]{Arazy}
    A bounded domain $\Omega$ is said to be symmetric if for every $a,b \in \Omega$ there exists an involution $\tau \in {\rm Aut}(\Omega)$ which interchanges $a$ and $b$.
\end{defn} \noindent The open unit disc $\mb D$, the unit polydisc $\mb D^n,$ the Euclidean ball $\mb B_n$ in $\mb C^n$ are some examples of bounded symmetric domains. The Hardy space $H^2(\Omega)$ on a bounded symmetric domain $\Omega$ is a well-studied function space \cite{Hahn-Mitchell:1969, Koranyi1965}. It is isometrically isomorphic to a closed subspace of $L^2(\del\Omega,d\Theta),$ where $d\Theta$ is the unique normalised $I_\Omega(0)$-invariant measure on the \v{S}ilov boundary $\del \Omega$ of $\Omega$ and $I_\Omega(0)=\{\phi \in {\rm Aut}(\Omega) : \phi(0)=0\}$ is the isotropy subgroup of $0$ in ${\rm Aut}(\Omega).$ 

\begin{defn}\cite{BD85, DS91}
    A proper holomorphic map $ \bl \pi:\Omega\to \widetilde{\Omega}\subseteq \mb C^n$ is {\it factored by automorphisms} if there exists a finite subgroup $G\subseteq {\rm Aut}(\Omega)$ such that for every $z\in \Omega,$
$\bl \pi^{-1}\bl \pi(z)=\cup_{\sigma\in G}\{\sigma(z)\}.$
\end{defn}  \noindent It is known that such a group $G$ is either a complex reflection group or conjugate to a complex reflection group. A finite complex reflection group $G$ is characterized by the fact that the ring of $G$-invariants polynomials in $n$ variables is a polynomial ring generated by some homogeneous system of polynomials $\{\theta_i\}_{i=1}^n$ associated to $G$ \cite[p. 282]{ST1954}. 
If a bounded domain $\Omega\subseteq\mb C^n$  is  a $G$-space, then the {\it basic polynomial} mapping
$\bl\theta=(\theta_1,\ldots,\theta_n):\Omega\to\bl\theta(\Omega),$
 is a proper holomorphic mapping factored by $G$ and $\bl\theta(\Omega)$ is a domain, see \cite{Rud1982, Try13}. Let $\widetilde{\Omega}\subseteq \mb C^n$ be an open set and $\phi: \Omega \to \widetilde{\Omega}$ be a proper map factored by $G,$ then $\widetilde{\Omega}$ is biholomorphic to $\bl\theta(\Omega)$ \cite[Proposition 4.4]{kag}. Henceforth, we work with $\bl\theta(\Omega)$ instead of the image of $\Omega$ under a proper holomorphic  map factored by $G.$ 
 
 An element $\sigma$ of $G$ acts on $\Omega$ by $\sigma \cdot z =\sigma^{-1}(z)$ and therefore on the Hardy space $H^2(\Omega)$ by $(\sigma \cdot f)(z)=f(\sigma^{-1} \cdot z).$ Under this action the Szeg\"o kernel $S_\Omega$ of $\Omega$ is $G$-invariant, that is, $$S_\Omega(\sigma \cdot z,\sigma \cdot w)=S_\Omega(z,w)\,\, \text{ for all } \sigma \in G \text{ and } z,w \in\Omega,$$ which makes the left regular representation $R: G\to \m U(H^2(\Omega))$ well-defined, $\m U(X)$ being the group of all unitary operators on the Hilbert space $X.$ Consequently, $H^2(\Omega)$ decomposes (canonical decomposition) into an orthogonal direct sum of isotypic components of the left regular representation of $G$ indexed by $\widehat G$ (the set of equivalence classes of irreducible representations of $G$). In Proposition \ref{idenh} and \ref{idenh2}, we prove the following:
\begin{itemize}
     \item  For each one-dimensional representation $\varrho\in \widehat G,$ we prove that the associated isotypic component is isometrically isomorphic to some analytic Hilbert module $H^2_\varrho(\bl\theta(\Omega))$ over the polynomial ring $\mb C[z_1,\ldots,z_n].$
 \end{itemize} In other words, we obtain a family 
\bea\label{family}\{H_\varrho^2(\bl \theta(\Omega)):\varrho\in\widehat G_1 \}\eea
of reproducing kernel Hilbert spaces, where $\widehat{G}_1$ is the equivalence classes of one-dimensional representations of $G.$ An analogous phenomenon was observed for the Bergman space on $\Omega,$ where the isotypic component related to the sign representation is isometrically isomorphic to the Bergman space of $\bl \theta(\Omega)$ \cite{kag}. On the basis of this observation, it is natural to define the Hardy space on $\bl \theta(\Omega)$ in the following manner. 

The \emph{sign representation} of a finite complex reflection group $G,$ $\sgn : G \to \mb C^*,$ is defined by \cite[p. 139, Remark (1)]{MR460484} \begin{eqnarray*} \sgn(\sigma) = (\det(\sigma))^{-1},\,\,\,\,\, \sigma\in G,\end{eqnarray*} see also Equation \eqref{sign}.
\begin{defn}\label{def}
    The reproducing kernel Hilbert space associated to the sign representation of $G$ in Equation \eqref{family} is said to be the \emph{Hardy space} on $\bl \theta(\Omega),$ denoted by $H^2(\bl \theta(\Omega)).$
\end{defn} For $\Omega = \mb D^n,$ Definition \ref{def} coincides with the one in \cite{Misra-SSR-Zhang} when $G$  is the permutation group on $n$ symbols and the one in \cite{GG23} when $G$ is a finite complex reflection subgroup of ${\rm Aut}(\mb D^n).$

Now we turn our attention to define Toeplitz operators on $H^2(\bl \theta(\Omega)).$ Noting that $G$ is a subgroup of $I_\Omega(0)$ and $d\Theta$ is a $I_\Omega(0)$-invariant measure  on $\del\Omega,$ we conclude that $d\Theta$ is also $G$-invariant on $\del\Omega$. It ensures that the (left) regular representation $R: G\to \m U(L^2(\del\Omega,d\Theta))$ is well-defined, and $L^2(\del\Omega,d\Theta)$ admits an orthogonal decomposition indexed by $\widehat G$ into the isotypic components associated to the regular representation. For every $\varrho\in\widehat G_1,$ the associated isotypic component of $L^2(\del\Omega)$ is isometrically isomorphic to some $L^2$-space with respect to some measure $d\Theta_\varrho$ supported on the \v{S}ilov boundary of $\bl\theta(\Omega),$ where the measure $d\Theta_\varrho$ is uniquely determined by the representation $\varrho.$ We denote it by $L^2(\del\bl \theta(\Omega),d\Theta_\varrho).$ Now one of our key findings states the following: \begin{itemize}
    \item Each $H_\varrho^2(\bl \theta(\Omega))$ can be realized as a closed subspace of $L^2(\del\bl \theta(\Omega),d\Theta_\varrho).$
\end{itemize}  For $u\in L^\infty(\del\bl \theta(\Omega)),$ the orthogonal projection $P_\varrho:L^2(\del\bl \theta(\Omega),d\Theta_\varrho)\to H_\varrho^2(\bl \theta(\Omega))$ induces the Toeplitz operator $T_u$ on $ H_\varrho^2(\bl \theta(\Omega))$  by 
$$T_uf=P_\varrho(uf).$$ We are now ready to present one of our main results. We refer to this characterization as a \emph{Brown-Halmos type characterization} in analogy with Theorem 6 of the  celebrated paper \cite{BH1964} by Brown and Halmos. 

\begin{thm}\label{bhthm} Let $m,p$ and $n$ be natural numbers such that $p|m$ and $\bl\theta$ be the basic polynomial map associated to the irreducible complex reflection group $G(m, p, n)$ which acts on $\mb D^n.$
    Suppose that $T :  H^2(\bl \theta(\mb D^n)) \to H^2(\bl \theta(\mb D^n))$ is a bounded linear operator. Then $T$ is a Toeplitz operator if and only if
    \Bea T^*_n T T_n&=&T\,\, \text{ ~and~ }\,\, \\ T^*_i T T^p_n &=&T T_{n-i}\text{~for~} i=1,\ldots,n-1,\Eea
    where $(T_1,\ldots,T_n)$ on $H^2(\bl \theta(\mb D^n))$ denotes the $n$-tuple of multiplication operators by the coordinate functions. 
\end{thm} 

An explicit description of $\bl \theta:\mb D^n \to \bl \theta(\mb D^n)$ is given in Equation \eqref{sym}. We highlight the generality of our framework by noting that every irreducible complex reflection group either belongs to the infinite family $G(m, p, n)$ indexed by three parameters, where $m,n,p$ are positive integers and $p$ divides $m,$  or, is one of 34 exceptional groups \cite{ST1954}. In particular, Theorem \ref{bhthm} recovers main results from \cite{BDS2021} and \cite{MR4244837} for $G(1,1,2)$ and $G(1,1,n)$, respectively.


 Our next goal is to establish certain multiplicative properties of Toeplitz operators. The first question to consider is: under what conditions the product of two Toeplitz operators is itself a Toeplitz operator? The following result demonstrates that, on $H^2(\bl \theta(\Omega)),$ this question is closely tied to the corresponding behavior on $H^2(\Omega).$  

\begin{thm}[Generalized zero-product property]\label{zerthm}
    Let the finite complex reflection group $G$ act on the bounded symmetric domain $\Omega$ and $\bl \theta: \Omega \to \bl \theta(\Omega)$ be a basic polynomial map associated to $G.$ Suppose that the $G$-invariant functions $\widetilde{u},\widetilde{v} \in L^\infty(\del\Omega)$ are of the form $\widetilde{u} = u \circ \bl \theta$ and $\widetilde{v} = v \circ \bl \theta.$ If $T_uT_v$ is a Toeplitz operator on $H^2_\mu(\bl \theta(\Omega))$ for some $\mu \in \widehat{G}_1,$ then \begin{itemize}
             \item[\rm (i)] $T_uT_v$ is a Toeplitz operator on $H^2_\varrho(\bl \theta(\Omega))$ for every $\varrho\in \widehat{G}_1.$
             \item[\rm (ii)] Moreover, $T_{\widetilde{u}}T_{\widetilde{v}}$ is a Toeplitz operator on $H^2(\Omega).$
         \end{itemize} Conversely, if $T_{\widetilde{u}}T_{\widetilde{v}}$ is Toeplitz operator on $H^2(\Omega),$ then so is $T_uT_v$ on $H^2_\varrho(\bl \theta(\Omega))$ for every $\varrho\in \widehat{G}_1.$
\end{thm}
In other words, we show that $T_uT_v$ is a Toeplitz operator on $H^2(\bl \theta(\Omega))$ if and only if $T_{\widetilde{u}}T_{\widetilde{v}}$ is Toeplitz operator on $H^2(\Omega).$ We illustrate an immediate application of Theorem \ref{zerthm}. Recall that on $H^2(\mb D)$, the product $T_uT_v$ is a Toeplitz operator if and only if either $u$ is co-analytic or $v$ is analytic \cite{BH1964}. An analogous, though more involved, result for $H^2(\mb D^2)$ can be found in \cite{GZ97}. Using Theorem \ref{zerthm} and \cite{GZ97}, we conclude such a characterization in Theorem \ref{semd2} for $H^2(D),$ $D$ being a proper holomorphic image of the bidisc $\mb D^2.$ An analogous phenomenon arises in the context of identifying commuting tuples of Toeplitz operators which we state below.
\begin{thm}(Commuting property)\label{zerthm1}
    With the same considerations as in Theorem \ref{zerthm}, if $T_uT_v=T_vT_u$  on $H^2_\mu(\bl \theta(\Omega))$ for some $\mu \in \widehat{G}_1,$ then \begin{itemize}
             \item[\rm (i)] $T_uT_v=T_vT_u$ on $H^2_\varrho(\bl \theta(\Omega))$ for every $\varrho\in \widehat{G}_1.$
             \item[\rm (ii)] Moreover, $T_{\widetilde{u}}T_{\widetilde{v}}=T_{\widetilde{v}}T_{\widetilde{u}}$ on $H^2(\Omega).$
         \end{itemize} Conversely, if $T_{\widetilde{u}}T_{\widetilde{v}}=T_{\widetilde{v}}T_{\widetilde{u}}$ on $H^2(\Omega),$ then $T_uT_v=T_vT_u$ on $H^2_\varrho(\bl \theta(\Omega))$ for every $\varrho\in \widehat{G}_1.$
\end{thm}
Recall that two Toeplitz operators on $H^2(\mb D)$ commute if and only if either both are analytic, or both are co-analytic, or one is a linear function of the other \cite{BH1964}. An analogous result for Toeplitz operators with bounded pluriharmonic symbols on $H^2(\mb B_n)$ can be found in \cite{Zheng98}. Combining \cite{Zheng98} and Theorem \ref{zerthm1}, we extend this conclusion to $H^2(D),$  $D$ being a proper holomorphic image of the unit ball $\mb B_n,$ cf. Theorem \ref{bncom}. 

The novelty of our work lies in the application of representation theory and the invariant theory of the groups of deck automorphisms associated with proper holomorphic maps. This approach allows us to study Toeplitz operators on the Hardy space of $\bl \theta(\Omega)$ without requiring any reference to the geometry of such domains. This framework opens up new possibilities for further research in operator theory and its connections to complex analysis. 

 \section{Preliminaries} We start this section by recalling some basic properties of proper holomorphic mappings which are of our interest.
\subsection{Proper holomorphic maps and complex reflection groups}
 Let $\Omega_1$ and $\Omega_2$ be two domains in $\mathbb C^n.$ A holomorphic map ${\bl \pi}:\Omega_1\to \Omega_2$ is said to be {\it proper} if ${\bl \pi}^{-1}(K)$ is a compact subset of $\Omega_1$ for every compact $K\subset \Omega_2.$ A proper holomorphic mapping
 $\bl \pi:\Omega_1\to \Omega_2$ is surjective and there exists a positive integer $m$ such that ${\bl \pi}:\Omega_1\setminus{\bl \pi}^{-1}({\bl \pi}(\mathcal J_{\bl \pi}))\to \Omega_2\setminus{\bl \pi}(\mathcal J_{\bl \pi})$  is a (unbranched) covering map with 
\begin{align*}
\text{cardinality of }{\bl \pi}^{-1}(w)&=m,\; \,\,w\in \Omega_2\setminus{\bl \pi}(\mathcal J_{\bl \pi}) \text{ and }\\
\text{cardinality of }{\bl \pi}^{-1}(w)&<m,\; \,\,w\in{\bl \pi}(\mathcal J_{\bl \pi}),
\end{align*}
where $\mathcal J_{\bl \pi}:=\{z\in \Omega_1:J_{\bl \pi}(z)=0\},$ $J_{\bl \pi}$ being the determinant of the complex jacobian matrix of ${\bl \pi}$ \cite[Chapter 15]{Rud1980}. We refer to $m$ as {\it the multiplicity of ${\bl \pi}$} and $\Omega_2$ as a {\it proper holomorphic image} of $\Omega_1.$

Let ${\rm Aut}(\Omega_1)$ be the group of all biholomorphic automorphisms of a domain $\Omega_1.$  An element $\sigma\in {\rm Aut}(\Omega_1)$ is called a {\it deck transformation} of the proper holomorphic mapping ${\bl \pi}:\Omega_1\to \Omega_2$ if ${\bl \pi}\circ \sigma={\bl \pi}$. The deck transformations of the proper holomorphic mapping $\bl \pi$ form a subgroup of ${\rm Aut}(\Omega_1)$ and we denote it by $\operatorname{Deck}({\bl \pi}).$ If a proper holomorphic map $ \bl \pi$ is factored by (automorphisms) $G,$ then $\operatorname{Deck}({\bl \pi})=G.$


In this article, our point of interest is the images of bounded symmetric domains under proper holomorphic mappings that are factored by automorphisms. 
E. Cartan completely classified the irreducible bounded symmetric domains (Cartan domains) in \cite{E35} (up to biholomorphisms). The list consists of four families of classical type domains and two exceptional domains of dimensions $16$ and $27.$ We collectively call them the {\it Cartan} domains. An excellent exposition on Cartan domains is due to Arazy \cite{Arazy}. 
Any bounded symmetric domain $D$ is of the form $D_1^{k_1} \times \cdots \times D_r^{k_r}$ for non-equivalent (non-biholomorphic) Cartan domains $D_i : i=1,\ldots,r.$ The unit ball with respect to the Euclidean norm in $\mb C^n,$ denoted by $\mb B_n,$ is an example of an irreducible bounded symmetric domain.  Rudin proved that every proper holomorphic mapping from $\mb B_n$ to some domain in $\mb C^n,\,n>1,$ is factored by some (automorphisms) $G$ \cite{Rud1982}. For $n>1,$ this result is extended for an irreducible bounded symmetric domain of classical type in $\mathbb C^n$ by Meschiari \cite[p. 18, Main Theorem]{Mes88}. Moreover, if $\Omega$ is a bounded symmetric domain, not necessarily irreducible and the multiplicity of $\bl \pi:\Omega \to \bl \pi(\Omega)$ is $2,$ then $\bl \pi$ is factored by (automorphisms) $G$ \cite{GZ2023}.  In each case, $G$ is either a finite complex reflection group or a conjugate to a finite complex reflection group. This  is, indeed, a general fact  for a proper holomorphic mapping factored by automorphisms, see \cite[Theorem 2.1, Theorem 2.2]{BD85}, \cite[Lemma 2.2, Theorem 2.5]{DS91}, \cite[Theorem 1.6]{Rud1982}, \cite[p. 506]{BB85}. Motivated by it, henceforth, we consider proper holomorphic mappings $\bl \pi:\Omega \to \bl \pi(\Omega),$ where $\Omega$ is a bounded symmetric domain and $\bl \pi$ is factored by (automorphisms) a finite complex reflection group  $G.$ Now we recall the definition of a complex reflection in $\mb C^n.$


\begin{defn}
{\it A complex reflection} on $\C^n$ is a linear homomorphism $\sigma: \C^n \rightarrow \C^n$ such that $\sigma$ is of finite order in $GL(n,\mb C)$ and the rank of $I_n - \sigma$ is 1, where $I_n$ is the identity operator on $\mb C^n.$ 
\end{defn}
In particular, if $\sigma$ is \textcolor{black}{of order $2$,} we call it a {\it reflection}. A group generated by complex reflections is called a {\it complex reflection group}. A complex reflection group $G$ acts on $\mb C^n$ by \bea\label{action1}\sigma \cdot z = \sigma^{-1}z \text{ for } \sigma \in G \text{ and } z \in \mb C^n.\eea 
\begin{ex}
    Let $G =\mathfrak S_n,$ the permutation group on $n$ symbols, acting on $\C^n$ by permuting  the coordinates, that is, $\sigma \cdot (z_1,\ldots,z_n) = (z_{\sigma^{-1}(1)},\ldots,z_{\sigma^{-1}(n)})$ for $\sigma \in \mathfrak{S}_n$ and $z_i\in \mb C$. The group $\mathfrak S_n$ is generated by transpositions $\{(i\,j)\}_{n \geq i>j \geq 1}$. Thus $\mathfrak S_n$ is a reflection group. One can realize $\mathfrak S_n$ in the following manner that aligns with above definition in a more appropriate way: consider the faithful representation 
\Bea
\rho: \mathfrak S_n \to GL(n, \C): (i\,j) \mapsto A_{(i\,j)},
\Eea
where $A_{(i\,j)}$ is the permutation matrix obtained by interchanging the $i$-th and the $j$-th columns of the identity matrix. 
\end{ex}

\begin{ex}\label{ex}Let $G = D_{2k} = \langle \delta, \sigma : \delta^k=\sigma^2 ={\rm Identity, \sigma \delta \sigma^{-1} = \delta^{-1}} \rangle$ be the dihedral group of order $2k.$ We define its action on $\C^2$ via the faithful representation $\rho$ defined by
\Bea
     \rho : G\to GL(2,\mb C): \delta\mapsto \begin{bmatrix}
    \zeta_k & 0\\
    0 & {\zeta_k}^{-1}
    \end{bmatrix}, \sigma\mapsto \begin{bmatrix}
    0 & 1\\
    1 & 0
    \end{bmatrix},
    \Eea
    where $\zeta_k$ denotes a primitive $k$-th root of unity. Writing the matrix representation of the group action with respect to the standard basis of $\C^2$
 we have $$G=\{\delta^j,\sigma\delta^j:j\in\{0,\ldots,k-1\}\},$$ where $\delta^j$ is a rotation having the eigenvalues ${\zeta_k}^{\pm j}$ and $\sigma\delta^j$ is a reflection having the eigenvalues $\pm 1.$  
\end{ex}

\subsubsection{Basic invariant polynomials} 
Chevalley-Shephard-Todd theorem states that  the ring of $G$-invariant polynomials in $n$ variables is equal to $\C[\theta_1,\ldots,\theta_n]$, where $\theta_i$'s are algebraically independent $G$-invariant homogeneous polynomials. These $\theta_i$'s are called basic invariant polynomials associated to $G.$ In \cite{Rud1982}, the mapping $\bl \theta:=(\theta_1,\ldots,\theta_n) : \mb C^n \to \mb C^n$ is said to be a basic polynomial mapping associated to the group $G.$ Let a  domain $\Omega \subseteq \mb C^n$ be a $G$-space, then $\bl \theta : \Omega \to \bl \theta(\Omega)$ is a proper holomorphic mapping with the deck automorphism group $G$ \cite{Rud1982, Try13}. 
In this paper, we refer to $$\bl \theta : \Omega \to \bl \theta(\Omega)$$ as a \emph{basic polynomial mapping} associated to the group $G.$ Moreover, any proper holomorphic map $\bl f : \Omega \to \Omega^\prime$ with the deck automorphism group $G$ is {\it isomorphic} to $\bl \theta$ (that is, $\bl f=\bl h \circ \bl \theta \circ \bl \psi$ for a biholomorphism $\bl h: \bl \theta(\Omega) \to \Omega^\prime$ and an automorphism $\bl \psi:\Omega \to \Omega$) and $\Omega^\prime$ is biholomorphic to $\bl \theta(\Omega)$ \cite[Proposition 4.4]{kag}.  Thus, the description of any proper holomorphic map $\bl f$ from $\Omega$ which is factored by (automorphisms) $G$ can be recovered from a basic polynomial map $\bl \theta:\Omega \to \bl \theta(\Omega)$ associated with $G$ (up to an isomorphism) and the proper image $\bl f (\Omega)$  is biholomorphic to $\bl \theta(\Omega).$ So we lose no generality if we work with a basic polynomial mapping associated to the finite complex reflection group $G$, instead of any proper holomorphic mapping factored by $G$. 

Lastly we note that the choice of a basic polynomial mapping associated to $G$ is not unique. Since any other basic polynomial mapping $\bl \theta^\prime:\Omega\to \bl \theta^\prime(\Omega)$ is isomorphic to $\bl \theta:\Omega\to\bl \theta(\Omega),$ our study is independent of the choice of $\bl \theta.$  Example \ref{symmet} explains it further.
\subsubsection{Proper holomorphic images of bounded symmetric domains}
We provide a few examples of the domains $\bl\theta(\Omega)$ on which our results are applicable. 


\begin{ex}\label{exo}
The irreducible finite complex reflection groups were classified by Shephard and Todd in \cite{ST1954}. They proved that every irreducible complex reflection group belongs to an infinite family $G(m, p, n)$ indexed by three parameters, where $ m,n,p$ are  positive integers and $p$ divides $m,$  or, is one of 34 exceptional groups. Although for certain values of $m,p$ and $n,$ $G(m,p,n)$ can be reducible, for example, $G(2,2,2)$ is the dihedral group of order $4$ which is isomorphic to the product of two cyclic groups of order $2.$ For a detailed study on $G(m, p, n),$ we refer to \cite[Chapter 2]{LT09}. 

Let $n>1.$ A set of basic invariant polynomials for the group $G(m,p,n)$ is given by elementary symmetric polynomials of $z_1^m,\ldots,z_n^m$ of degrees $1,\ldots,n-1$ and $(z_1\cdots z_n)^q,$ where $q=m/p$ \cite[p. 36]{LT09}. We denote the elementary symmetric polynomials of degree $i$ of $z_1^m,\ldots,z_n^m$ by $\theta_i(z)$ for $i=1,\ldots,n-1$ and $\theta_n(z) = (z_1\cdots z_n)^q.$ 
The group $G(m, p, n)$ has an action on $\mb D^n$ as given in Equation \eqref{action1}. Thus we have an explicit description for the basic polynomial map $\bl \theta:=(\theta_1,\ldots,\theta_n) : \mb D^n \to \bl\theta(\mb D^n)$ associated to $G(m,p,n).$ Any image of $\mb D^n$ under the proper holomorphic mapping $\bl \pi$ with $\operatorname{Deck}({\bl \pi})=G(m,p,n),\,n>1,$ is biholomorphic to  $\bl \theta(\mb D^n).$ 
\end{ex}

\begin{ex}\label{symmet} Let $\mathfrak{S}_n$ denote the permutation group on $n$ symbols and
\Bea s_i(z_1,\ldots ,z_n) = \sum_{1\leq k_1< k_2 <\ldots <k_i\leq n} z_{k_1} \cdots z_{k_i}\Eea
be the $i$-th elementary symmetric polynomial in $n$ variables. The symmetrization map 
$$\bl s := (s_1,\ldots,s_n): \mb D^n \to \bl s(\mb D^n)$$ 
is a proper holomorphic map factored by $\mathfrak{S}_n.$   The domain $\mb G_n:=\bl s(\mb D^n),$ is called   the symmetrized polydisc \cite{Cos05}. It is well-known that the permutation group $\mathfrak{S}_n$ is equal to $G(1,1,n).$
    The symmetrization map $\bl s = (s_1,\ldots,s_n): \mb D^n \to \mb G_n$ is a basic polynomial associated to $\mathfrak{S}_n$ and coincides with the map $\bl \theta$ described in Example \ref{exo} for $G(1,1,n)$.
    
    Let the power sum symmetric polynomial of degree $k$ in $n$ variables be denoted by $$p_k(z_1,\ldots,z_n) = \sum_{i=1}^n z_i^k.$$ Then $\bl \theta:=(p_1,\ldots,p_n):\mb D^n \to \bl \theta(\mb D^n)$ is also a basic polynomial map associated to $\mathfrak{S}_n.$ The domains $\bl \theta(\mb D^n)$ and the symmetrized polydisc $\mb G_n$ are biholomorphic to each other. For example, $\bl h : \mb G_2 \to \bl \theta(\mb D^2)$ is a biholomorphism given by $\bl h(s_1,s_2)=(s_1,s_1^2-2s_2).$ \end{ex}
    \begin{ex}\label{dk}
       The group $D_{2k}=G(k,k,2)$ acts on $\mb D^2$ (cf. Example \ref{ex}) and $\bl \theta :=(\theta_1,\theta_2): \mb D^2 \to \bl \theta(\mb D^2)$ is a basic polynomial map where $\theta_1(z_1,z_2)=z_1^k+z_2^k$ and $\theta_2(z_1,z_2)=z_1z_2.$ We denote the domain $\bl \theta(\mb D^2)$ by $\mathscr{D}_{2k}.$ 
    \end{ex}

\begin{ex}\label{ball}
 For positive integers $m,n>1,$ let $\mathcal E_n(m):=\{z\in\mathbb C^n:|z_1|^{2/m}+|z_2|^2+\ldots+|z_n|^2<1\}$ denote the complex ellipsoid. For fixed $n,m>1,$ the mapping $\phi_{n,m} :\mathbb B_n \to \mathcal E_n(m),$ defined by $$ \phi_{n,m}(z_1,z_2,\ldots,z_n)= (z_1^m,z_2,\ldots,z_n),$$ is a basic polynomial map associated to $\mb Z_m$ (the cyclic group of order $m$).   
\end{ex}
\begin{ex}\label{car2}
    The {\it classical Cartan domains of type $III$}, denoted by $\m R_{III}(n)$, is the set of all $n \times n$ symmetric (complex) matrices $A$ for which $\mathbb I_n-AA^*$ is positive definite \cite[p. 9]{Arazy}. Let $\m R_{III}(2)$ be the classical Cartan domain of third type of rank $2$. The proper holomorphic map $\bl \theta: \m R_{III}(2) \to \bl \theta(\m R_{III}(2))$ defined by $$\bl \theta(z_1,z_2,z_3) = (z_1,z_2,z_3^2-z_1z_2),$$ is a basic polynomial map associated to the cyclic group of order $2,$ $\mb Z_2.$ The domain $\mb E:=\bl \theta(\m R_{III}(2)),$ is called  the tetrablock \cite{AWY07}.
\end{ex}
\begin{ex}\label{lieball}  The {\it classical Cartan domains of type $IV$} (alternatively, {\it the Lie ball} $L_n$) is the following domain :
\begin{eqnarray*}
L_n:=\left\{z\in\mathbb \mathcal R_I(1\times n): \sqrt{\left(\sum_{j=1}^n|z_j|^2\right)^2-\left|\sum_{j=1}^nz_j^2\right|^2}<1-\sum_{j=1}^n|z_j|^2\right\}.
 \end{eqnarray*}
    For $n\geq 2,$ we define the proper holomorphic mapping of multiplicity $2$ by $\bl \Lambda_n : L_n \to \bl \Lambda_n (L_n):= \mathbb L_n$ for
 \begin{eqnarray*}
     \bl \Lambda_n(z_1,z_2,\ldots,z_n)=(z_1^2,z_2,\ldots,z_n).\end{eqnarray*}
This is a basic polynomial map associated to $\mb Z_2$ on $L_n.$ Moreover, we know that $L_2$ is biholomorpic to $\mathbb D^2$ and $L_3$ is biholomorphic to $\mathcal R_{III}(2).$ This leads to the observation that $\mb L_2$ is biholomorphic to the symmetrized bidisc $\mb G_2$ and $\mb L_3$ is biholomorphic to the tetrablock $\mb E$ \cite[Corollary 3.9]{GZ2023}.
\end{ex}

Proper holomorphic images of irreducible bounded symmetric domains can be described (up to biholomorphisms) using \cite[Theorem 3]{Gottschling1969}, \cite[Propostion 3.3]{GZ2023} and \cite[p. 18, Main Theorem]{Mes88}. The same course of action will not work for reducible bounded symmetric domains. For a (reducible or irreducible) bounded symmetric domain  $\Omega,$ a description for all possible  complex reflections in $\Aut(\Omega)$ is given in \cite[p. 702, Theorem 2]{Gottschling1969}. Making use of this observation,  a classification for all possible images (up to a biholomorphism) of bounded symmetric domains under a proper holomorphic mapping with multiplicity $2$ is obtained in \cite[proposition 3.6]{GZ2023}. 

\subsubsection{\v{S}ilov Boundary} We recall the definition of \v{S}ilov boundary of a domain from \cite{Gamelin1969}.
\begin{defn}
The \v{S}ilov boundary $\partial \Omega$ of a bounded domain $\Omega$ is given by the
closure of the set of its peak points and a point $w \in \overbar{\Omega}$ is said to be a peak point of $\Omega$ if there exists a function $f \in \m A(\Omega)$ such that $|f(w)|>|f(z)|$ for all $z \in \overbar{\Omega} \setminus \{w\},$ where $\m A(\Omega)$ denotes the algebra of all functions holomorphic on $\Omega$ and continuous on $\overbar{\Omega}.$
\end{defn} 
For example, the \v{S}ilov boundary of the polydisc $\mb D^n$ is the $n$-torus $\mb T^n.$ The \v{S}ilov boundary of the unit ball $\mb B_n$ coincides with its topological boundary. Since $\bl \theta : \Omega \to \bl \theta(\Omega)$ is a proper holomorphic map which can be extended to a proper holomorphic map of the same multiplicity from $\Omega'$ to $\bl \theta(\Omega)',$ where the open sets $\Omega'$ and $\bl \theta(\Omega)'$ contain $\ov{\Omega}$ and $\ov{\bl \theta(\Omega)},$ respectively. Then \cite[p. 100, Corollary 3.2]{KZ2013} states that $\bl \theta^{-1}(\del\bl \theta(\Omega)) = \del\Omega.$ Thus \begin{eqnarray}\label{shilovboundary} \del\bl \theta(\Omega) = \bl \theta(\del\Omega).\end{eqnarray}
For instance, the \v{S}ilov boundary of the symmetrized polydisc $\bl s(\mb D^n)$ is given by $\bl s(\mb T^n).$ The \v{S}ilov boundary of $\mb L_n$ is $\Lambda_n(\del L_n)$ (cf. Example \ref{lieball}) \cite[Proposition 4.1]{GZ2023}.

\subsection{Hardy space on  bounded symmetric domains}\label{hardysubsec}  
 A notion of the Hardy space on a bounded symmetric domain $\Omega$ is given in \cite[p. 521]{Hahn-Mitchell:1969}. We reproduce it here for the sake of completeness of our exposition. Recall that $I_\Omega(0)$ denotes the isotropy subgroup of $0$ in $\Aut(\Omega).$ The group $I_\Omega(0)$ acts transitively on the \v{S}ilov boundary $\del \Omega$. There exists a unique normalised $I_\Omega(0)$-invariant measure on $\del \Omega$, say $d\Theta$. The $L^2$-space $L^2(\del \Omega):=L^2(\del \Omega,d\Theta)$ is the Hilbert space of complex measurable functions on $\del \Omega$ with the inner product $$\inner{f}{g}_{L^2} = \int_{\del \Omega}f(t)\ov{g(t)} d\Theta(t),\,\, f,g\in L^2(\partial\Omega).$$
The action of the group $I_\Omega(0)$  on $L^2(\del\Omega)$ is given by $\sigma(f)(z)= f(\sigma^{-1} \cdot z)$ for $\sigma \in I_\Omega(0)$ and $f \in L^2(\del\Omega).$ Since the measure $d\Theta$ is $I_\Omega(0)$-invariant,  for any $\sigma \in I_\Omega(0),$ it follows that
\bea \label{uni}  \inner{\sigma(f)}{\sigma(g)} &=& \int_{\del \Omega}f(\sigma^{-1}\cdot t)\ov{g(\sigma^{-1}\cdot t)} d\Theta(t)=\int_{\del \Omega}f(t)\ov{g(t)} d\Theta(t)=\inner{f}{g}.\eea

Let $\m O(\Omega)$ denote the algebra of holomorphic functions on $\Omega.$ The Hardy space $H^2(\Omega)$ is defined by 
\Bea H^2(\Omega) : = \{f\in \m O(\Omega) :\Vert f\Vert_{H^2}:= \sup_{0 < r <1} \left( \int_{\del \Omega}|f(rt)|^2 d\Theta(t)\right)^{1/2} < \infty \} \Eea 
in \cite{Hahn-Mitchell:1969}. For every function $f\in H^2(\Omega),$ its radial limit $\widetilde{f}$ exists almost everywhere (with respect to $\Theta$) on $\del\Omega,\,\,  \widetilde f\in L^2(\del \Omega)$ and $\norm{f}_{H^2} = \Vert\widetilde{f}\Vert_{L^2}$ \cite[p. 126]{Up96}. We identify $f$ and $\widetilde{f},$ henceforth,  no distinction will be made between these two realizations. Let $\widetilde{P}$ be the orthogonal projection of $L^2(\del \Omega)$ onto $H^2(\Omega).$ Thus, there is an embedding of $H^2(\Omega)$ into $L^2(\del \Omega)$ as a closed subspace \cite{Koranyi1965}, \cite[p. 526, Theorem 6]{Hahn-Mitchell:1969}.   Moreover, $H^2(\Omega)$ is a reproducing kernel Hilbert space and its reproducing kernel $S_\Omega$  is referred as the Szeg\"o kernel of $\Omega.$   
For every $w \in \Omega,$ the holomorphic function $S_\Omega(\cdot,w)$ is in $H^2(\Omega).$ Following \cite[p. 126]{Up96}, we have 
\bea\label{main} (\widetilde{P}f)(z) = \inner{f}{S_\Omega(\cdot,z)}_{L^2}\eea 
for every $f \in L^2(\del \Omega) ,\ \Omega$ being an irreducible bounded symmetric domain. A proof can be found in \cite[Section 2.9]{Up96}, see also \cite{FK90}.  
Further, if $\Omega=\prod_{i=1}^n\Omega_i,$ where each $\Omega_i$ is an irreducible bounded symmetric domain, then $H^2(\Omega)$ can be naturally identified with $\otimes_{i=1}^nH^2(\Omega_i).$ The Szeg\"o kernel $S_\Omega$ of $H^2(\Omega)$ is taken to be the reproducing kernel $\prod_{i=1}^n S_{\Omega_i}$ of 
$\otimes_{i=1}^nH^2(\Omega_i),$ that is,
\Bea
S_\Omega(z,w)=\prod_{i=1}^nS_{\Omega_i}(z_i, w_i),
\Eea
where $z_i,w_i\in\Omega_i$ for $i=1,\ldots,n$ and  Equation \eqref{main} holds.

\subsection{Toeplitz Operators}For $u \in L^\infty(\del\Omega),$ the Laurent operator $M_{u} : L^2(\del\Omega) \to L^2(\del\Omega)$ and the Toeplitz operator $T_{u}: H^2(\Omega)\to H^2(\Omega)$  are defined by
$$M_{u}f=u f  \text{~and~}  T_{u} = \widetilde{P}M_{u},$$ respectively,  $\widetilde{P}$ being the orthogonal projection from $L^2(\del\Omega)$ onto  $H^2(\Omega).$ From Equation \eqref{main}, we have 
\bea\label{toepo} (T_{u}f)(z) = \inner{uf}{S_\Omega(\cdot,z)}_{L^2}.\eea
 We prove a lemma, one of whose immediate consequence is the fact that the linear map $u\mapsto T_u$ from $L^\infty(\partial\Omega)$ into $\m B(H^2(\Omega))$ is isometric, $\m B(H^2(\Omega))$ being the algebra of all bounded operators on $H^2(\Omega).$ We follow the ideas from the proof of \cite[Theorem 2.1]{DJ77}.

\begin{lem}\label{irr}
    Suppose that $\Omega$ is an irreducible bounded symmetric domain in $\mb C^n$ and  the Toeplitz operator $T_u$ is invertible in $\m B(H^2(\Omega)).$   Then $u$ is invertible in $L^\infty(\del\Omega).$
\end{lem} 
\begin{proof}
    Let  $h(z,w):=\inner{z}{w}$ and $\psi$ be a non-negative measurable function on $\mb C^n.$ For $z\in \del\Omega,$ let $F(\xi)=\psi(h(z,\xi))$ for $\xi\in\del\Omega.$   Since $d\Theta$ is $ I_\Omega(0)$-invariant,  $\int_{\del \Omega} F(\xi) d\Theta(\xi)$ is independent of $z.$ For $k\geq 1,$ let $$a_k = \int_{\del \Omega} |1+h(z,\xi)|^{2k} d\Theta(\xi)$$ and note that $a_k$ is independent  of $z.$

We observe that for a fixed $z\in \del\Omega,$  the function $h(z,w)$ has the only peak point at $w=z$ in the \v{S}ilov boundary of $\Omega.$ That is, $h(z,z)=1$ for $z\in \del\Omega$  and  
$$  |h(z,w)|<1 \text{~~for every~~}w \in \del\Omega,\,\ w\neq z. $$ So there exists a neighbourhood $U$  of $z$ in $\del\Omega$ such that 
    $$\frac{1}{a_k} \int_{\del \Omega\setminus U} |1+h(z,\xi)|^{2k} d\Theta(\xi) \rightarrow 0 \text{ as } k \rightarrow \infty.$$
     It follows that 
    $$\frac{1}{a_k}\int_{\del \Omega} g(\xi)|1+h(z,\xi)|^{2k} d\Theta(\xi) \rightarrow g(z) \text{~~as~~} k \rightarrow \infty$$ for a continuous function $g$ on $\del\Omega.$ Since $T_u$ is invertible, there exists an $\epsilon>0$ such that 
   
   $$\norm{T_uf}\geq\epsilon \norm{f} \text{~~for every~~} f \in H^2(\Omega).$$
   In particular, for $f_k(z)=(1+h(z,\xi))^{k}$ this gives $\norm{T_uf_k}^2\geq \epsilon^2 a_k.$ For any positive valued continuous function $g,$ it follows that $$\frac{1}{a_k} \int_{\del \Omega}\int_{\del \Omega} |u(z)|^2 g(\xi) |1+h(z,\xi)|^{2k} d\Theta(\xi)d\Theta(z)\geq \epsilon^2 \int_{\del \Omega} g(\xi) d\Theta(\xi).$$ 
   An application of Fubini's theorem yields
   $$\int_{\del \Omega} |u(z)|^2 g(z) d\Theta(z)\geq \epsilon^2  \int_{\del \Omega} g(\xi) d\Theta(\xi).$$
   Since this inequality holds for every positive continuous function $g,$ we have $|u(z)|^2\geq \epsilon^2$ almost everywhere on $\del\Omega.$ This completes the proof. 
\end{proof}
As a consequence, we have the following corollary. A proof along the line of \cite{Douglas98} is included for the sake of completeness.
\begin{cor}\label{irrcor}
If $\Omega$ is an irreducible bounded symmetric domain, then $u\mapsto T_u$ is a $*$-linear isometry of $L^\infty(\del\Omega)$ into $\m B(H^2(\Omega)).$
\end{cor}
\begin{proof}
We only prove $\Vert T_u\Vert=\Vert u\Vert_\infty,$ as the proof of $*$-linearity is trivial.
    Since $T_u-\l=T_{u-\l}$ for $\l\in\mb C,$ it follows from Lemma \ref{irr} that ${\mr{Spec}}(M_u)\subseteq {\mr{Spec}}(T_u),$ here ${\mr{Spec}}(T)$ denotes the spectrum of $T.$ Thus, $$\m R(u)={\mr{Spec}}(M_u)\subseteq {\mr{Spec}}(T_u),$$ where $\m R(u)$ is the essential range of $u.$ Therefore,
$$ \Vert u\Vert_\infty\geq\Vert T_u\Vert\geq \text{spectral radius of}~~ T_u\geq\sup \{\vert \l\vert: \l\in \m R(u)\}=\Vert u\Vert_\infty.$$
This completes the proof.
\end{proof}

 \subsection{Orthogonal decomposition and projection operators} Let $G$ be a finite complex reflection group which is a subgroup of $\Aut(\Omega).$ Since every complex reflection fixes the origin,  $G \subset I_\Omega(0).$ 
 For  $\sigma\in G,$ the linear map $R_\sigma: L^2(\del\Omega) \to L^2(\del\Omega)$ is defined by 
 \bea\label{rsigma}R_\sigma(f) = \sigma(f)=f \circ \sigma^{-1}.\eea 
 Equation \eqref{uni} implies that each $R_\sigma$ is well-defined and the map $R: \sigma\mapsto R_\sigma$ is a unitary representation of $G$ on $L^2(\del\Omega).$
 
 Let $\widehat{G}$ denote the set of all equivalence classes of irreducible representations of $G.$ For $\varrho \in \widehat{G},$ the linear operator $\mb P_\varrho:L^2(\del\Omega) \to L^2(\del\Omega)$ defined by
\Bea \mb P_\varrho \phi = \frac{\deg\varrho}{|G|}\sum_{\sigma \in G} \chi_\varrho(\sigma^{-1})R_\sigma(\phi), \Eea 
   is an idempotent \cite[p. 24, Theorem 4.1]{MR2553682}, where $\chi_\varrho$ denotes the character of $\varrho,$ $\deg\varrho$ is the degree of the representation $\varrho$ and $|G|$ is the order of the group $G.$ In fact,  for $\sigma \in G,\,\, R_\sigma$ is a unitary by Equation \eqref{uni}, it follows that $R_\sigma^*=R_{\sigma^{-1}}.$ Moreover, for every $\varrho \in \widehat{G},$ $\ov{\chi_\varrho(\sigma^{-1})}=\chi_\varrho(\sigma)$ \cite[p. 15, Proposition 2.5]{MR2553682}. Hence $\mb P_\varrho=\mb P_\varrho^*.$ So $\mb  P_\varrho$ is an orthogonal projection for every $\varrho\in\widehat G.$ 

Since $\oplus_{\varrho \in \widehat{G}} \mb P_\varrho = I_{L^2(\del\Omega)},$ $L^2(\del\Omega)$ is an orthogonal direct sum as follows: \bea\label{l2ortho} L^2(\del\Omega) = \oplus_{\varrho \in \widehat{G}} \mb P_\varrho (L^2(\del\Omega)). \eea

\begin{ex}\label{onb}\rm  Let $\mb T=\{z\in \mb C: |z|=1\}.$ The \v{S}ilov boundary of the open unit polydisc $\mb D^n$ is the $n$-torus $\mb T^n$ and the set $\{z^{\bl \alpha}=\prod_{i=1}^n z_i^{\alpha_i} : \bl \alpha \in \mb Z^n\}$ forms an orthogonal basis for $L^2(\mb T^n).$ Recall that the permutation group $\mathfrak{S}_n$ acts on $\mb T^n$ by $\sigma \cdot (z_1,\ldots,z_n) = (z_{\sigma^{-1}(1)},\ldots,z_{\sigma^{-1}(n)})$ for $\sigma \in \mathfrak{S}_n$ and $z_i\in \mb T$. Moreover, $R_{\sigma}(z^{\bl \alpha})=\sigma(z^{\bl \alpha})= \prod_{i=1}^n z_{\sigma(i)}^{\alpha_i}=\prod_{i=1}^n z_i^{\alpha_{\sigma^{-1}(i)}}=z^{\sigma \cdot \bl \alpha}.$
\begin{itemize}[leftmargin=*]
    \item Let $\bl \beta \neq \sigma \cdot \bl \alpha$ for all $\sigma \in \mathfrak{S}_n.$ Then $\mb P_\varrho(z^{\bl \beta})$ and $\mb P_\varrho(z^{\bl \alpha})$ are mutually orthogonal.
    \item If $\bl \beta = \sigma \cdot \bl \alpha$ for some $\sigma \in \mathfrak{S}_n$ then $\mb P_\varrho z^{\bl \beta} = \chi_\varrho(\sigma) \mb P_\varrho z^{\bl \alpha}.$ Further, if $\varrho \in \widehat{\mathfrak{S}}_n$ is not equivalent to the trivial representation, there exists at least one $\sigma_0 \in \mathfrak{S}_n$ such that $\chi_\varrho(\sigma_0) \neq 1.$ Let $\bl \alpha \in \mb Z^n$ be such that $ {\sigma_0 \cdot \bl \alpha} = {\bl \alpha}.$    Then $\mb P_\varrho(z^{\bl \alpha})=\mb P_\varrho (z^{\sigma_0 \cdot \bl \alpha}) = \chi_\varrho(\sigma_0) \mb P_\varrho(z^{\bl \alpha}).$ Consequently, $\mb P_\varrho(z^{\bl \alpha})= 0.$
    \item For example, the transposition $\sigma = (1~2)$ in $\mathfrak{S}_3$ keeps the multi-index $\bl \alpha =(1,1,4)$ fixed and the character $\chi_{\rm sgn}((1~2))=-1$ (see Equation \eqref{sign} for details on the  sign representation). Therefore, $\mb P_{\rm sgn}(z^{\bl \alpha})=0.$
    \item 

For a representation $\varrho \in \widehat{\mathfrak{S}}_n,$ let $$I_\varrho:=\{\bl\alpha\in\mb Z^n: \mb P_\varrho(z^{\bl \alpha}) \neq 0 \} \text{~~and~~}  [\bl\alpha]:=\{\sigma\cdot\bl\alpha:\sigma\in\mathfrak S_n\} \text{~~for~~}\bl\alpha\in \mb Z^n.$$
    Clearly, $ \{[\bl\alpha]:\bl\alpha\in I_\varrho\}$ is a partition of $I_\varrho$ into equivalence classes, namely, the orbits of elements in $ I_\varrho$ under the action of $\mathfrak S_n.$  The subset $\{\mb P_\varrho z^{\bl \alpha} : \bl \alpha \in I_\varrho\}$ forms an orthogonal basis for $\mb P_\varrho(L^2(\mb T^n)),$ here $\bl\alpha$ stands for any representative of the orbit  $[\bl\alpha]$ of $\bl\alpha.$

\end{itemize}   
   
\end{ex}

The subspace $ H^2(\Omega)\subset L^2(\del\Omega)$ is left invariant by $R_\sigma.$ Its restriction to $ H^2(\Omega),$ also denoted by $R_\sigma,$ is a unitary operator on $ H^2(\Omega). $
Thus, the map $R: \sigma\mapsto R_\sigma$ is a unitary representation of $G$ on $H^2(\Omega).$ 
For every $\varrho \in \widehat{G},$ the linear map $\mb P_\varrho:H^2(\Omega) \to H^2(\Omega),$ defined by
\Bea \mb P_\varrho \phi = \frac{\deg\varrho}{|G|}\sum_{\sigma \in G} \chi_\varrho(\sigma^{-1})R_\sigma(\phi), \Eea 
is an orthogonal projection onto the isotypic component associated to the irreducible representation $\varrho$ in the decomposition of the regular representation of $G$ on $H^2(\Omega)$ \cite[p. 24, Theorem 4.1]{MR2553682} \cite[Corollary 4.2]{BDGS22} and
\bea \label{deco} H^2(\Omega) = \oplus_{\varrho \in \widehat{G}} \mb P_\varrho (H^2(\Omega)). \eea 
Moreover, $\mb P_\varrho(H^2(\Omega))$ is a closed subspace of $H^2(\Omega)$ and the reproducing kernel $S_\varrho$ of $\mb P_\varrho(H^2(\Omega))$ is given by 
\bea \label{repvar}S_\varrho(z, w) =(\mb P_\varrho S_\Omega)(z,w)=\frac{1}{|G|}\sum_{\sigma \in G} \chi_{\varrho}(\sigma^{-1}) S_\Omega(\sigma^{-1} \cdot z, w).\eea \begin{rem}\label{prho} We  emphasize that such an orthogonal decomposition of $H^2(\Omega)$ in Equation \eqref{deco} is possible here since the measure $d\Theta$ is $G$-invariant. In the sequel, we show that each $\mb P_\varrho (H^2(\Omega))$ is isometrically isomorphic to some reproducing kernel Hilbert space on $\bl \theta(\Omega)$ and whence define a notion of Hardy space on $\bl \theta(\Omega).$ Clearly, this approach may not work in general. \end{rem} 
For  $f \in \mb P_\varrho(L^2(\del\Omega)),$ it follows from Equation 
\eqref{main} that $$(\widetilde{P}f)(z) = \inner{f}{S_\Omega(\cdot,z)}_{L^2}=\inner{\mb P_\varrho f}{S_\Omega(\cdot,z)}_{L^2}=\inner{f}{S_{\varrho}(\cdot,z)}_{L^2}.$$ Hence $\widetilde{P}f \in \mb P_\varrho(H^2(\Omega)).$ Let $\widetilde{P}_\varrho :\mb P_\varrho(L^2(\del\Omega)) \to \mb P_\varrho(H^2(\Omega))$ be the orthogonal projection. We note that $\widetilde{P}_\varrho=\widetilde{P}\mb P_\varrho$ and thus 
$$(\widetilde{P}_\varrho f)(z) = \inner{f}{S_{\varrho}(\cdot,z)}_{L^2}.$$
If $u\in L^\infty(\del\Omega)$ is $G$-invariant and  $f \in \mb P_\varrho(H^2(\Omega)),$ then   $u f \in \mb P_\varrho(L^2(\del\Omega))$ and
\bea \label{subToep}(T_{u}f)(z)= (\widetilde{P}(M_{u}f))(z) = \inner{uf}{S_\Omega(\cdot,z)}_{L^2}=  \inner{uf}{S_{\varrho}(\cdot,z)}_{L^2}=\widetilde{P}_\varrho(uf).\eea

\subsubsection{One-dimensional representations}
Since the one-dimensional representations of $G$ play an important role in our discussion, we elaborate on some relevant results for the same. We denote the set of equivalence classes of the one-dimensional representations of $G$ by $\widehat{G}_1.$

A hyperplane $H$ in $\C^n$ is called reflecting if there exists a complex reflection in $G$ acting trivially on $H.$ For a complex reflection $\sigma \in G,$ let $H_{\sigma} := \ker(I_n - \sigma).$ By definition, the subspace $H_{\sigma}$ has dimension $n-1.$ Clearly, $\sigma$ fixes the hyperplane $H_{\sigma}$ pointwise. Hence each $H_\sigma$ is a reflecting hyperplane.  By definition, $H_\sigma$ is the zero set of a non-zero homogeneous linear polynomial $L_\sigma$ on $\C^n$, determined up to a non-zero constant multiple, that is, 
$$H_\sigma = \{z\in\C^n: L_\sigma(z) = 0\}.$$ 
Moreover, the elements of $G$ acting trivially on a  reflecting hyperplane forms a cyclic subgroup of $G.$

 Let $H_1,\ldots, H_t$ be the distinct reflecting hyperplanes associated to the group $G$ and  the corresponding cyclic subgroups be $G_1,\ldots, G_t,$ respectively. Suppose $G_i = \langle a_i \rangle$ and the order of each $a_i$ is $m_i$ for $i=1,\ldots,t.$ For every one-dimensional representation $\varrho$ of $G,$ there exists a unique $t$-tuple of non-negative integers $(c_1,\ldots,c_t),$ where $c_i$'s are the least non-negative integers that satisfy the following: \begin{eqnarray}\label{ci}\varrho(a_i) =\big( \det(a_i)\big)^{c_i}, \,\, i=1,\ldots,t.\end{eqnarray} The $t$-tuple $(c_1,\ldots,c_t)$ solely depends on the representation $\varrho.$ 

For $\varrho \in \widehat{G}_1,$ the character of $\varrho,$ $\chi_\varrho : G \to \mb C^*$ coincides with the representation $\varrho.$ The set of elements of $H^2(\Omega)$ relative to the one-dimensional representation $\varrho$ is given by 
\begin{eqnarray}\label{invar} R^G_{\varrho}(H^2(\Omega)) = \{f \in H^2(\Omega) : \sigma(f) = \chi_\varrho(\sigma) f ~ {\rm for ~~ all~} \sigma \in G\}.\end{eqnarray}  
The elements of the subspace $R^G_{\varrho}(H^2(\Omega))$ are said to be  $\varrho$-invariant functions. We recall a lemma concerning the $\varrho$-invariant functions which is going to be useful in the sequel. 
\begin{lem}\label{gencz}\cite{kag}
Suppose that the linear polynomial $\ell_i$ is a defining function of $H_i$ for $i=1,\ldots,t$ and   
$$\ell_\varrho = \prod_{i=1}^t \ell_i^{c_i }$$ is a homogeneous polynomial, where $c_i$'s are unique non-negative integers as described in Equation \eqref{ci}. Any element $f \in R^G_{\varrho}(H^2(\Omega))$ can be written as $f = \ell_\varrho \cdot (\widetilde{f} \circ \bl \theta)$ for a holomorphic function $\widetilde{f}$ on $\bl \theta(\Omega).$ 
\end{lem}
  The \emph{sign representation} of a finite complex reflection group $G,$ $\sgn : G \to \mb C^*,$ is defined by \cite[p. 139, Remark (1)]{MR460484} \begin{eqnarray}\label{sign} \sgn(\sigma) = (\det(\sigma))^{-1},\,\,\,\,\, \sigma\in G.\end{eqnarray} Moreover, we note from Equation \eqref{ci} that $${\rm sgn} (a_i) = (\det(a_i))^{-1}= \big(\det(a_i)\big)^{m_i-1},\,\, i=1,\ldots,t,$$
  which implies the following corollary of Lemma \ref{gencz}. 

\begin{cor}\cite[p. 616, Lemma]{MR117285}\label{Jac}
Let $H_1,\ldots, H_t$ denote the distinct reflecting hyperplanes associated to the group $G$ and let $m_1,\ldots, m_t$ be the orders of the corresponding cyclic subgroups $G_1,\ldots, G_t,$ respectively. Then 
\begin{eqnarray*}
 \ell_{\rm sgn}(z) = J_{\bl \theta} (z) = c \prod_{i=1}^t \ell_i^{m_i -1 }(z) ,
\end{eqnarray*}
where $J_{\bl \theta}$ is the determinant of the complex jacobian matrix of the basic polynomial map $\bl \theta$ and $c$ is a non-zero constant.
\end{cor}
Generalizing the notion of a relative invariant subspace, defined in Equation \eqref{invar}, we define the relative invariant subspace of $L^2(\del\Omega)$ associated to a one-dimensional representation $\varrho$ of $G,$ by
\Bea
R^G_{\varrho}(L^2(\del\Omega)) = \{f \in L^2(\del\Omega) : \sigma(f)  = \chi_\varrho(\sigma) f ~ { \rm a.e. ~~ for ~~ all~} \sigma \in G \}.
\Eea

\begin{rem}\label{rem3}
    We note that for every $\varrho \in \widehat{G}_1,$ 
    \begin{enumerate}[leftmargin=*]
        \item[\rm 1.] $R^G_{\varrho}(L^2(\del\Omega))=\mb P_{\varrho}(L^2(\del\Omega)).$ Since $\ell_\varrho$ vanishes only on a set of measure zero,  any $f \in \mb P_\varrho(L^2(\del\Omega))$ can be written as $f = \widehat{f} \ell_\varrho,$ where $\widehat{f} = \frac{f}{\ell_\varrho}.$ Clearly, $\widehat{f}$ is $G$-invariant. Hence we write $\widehat{f} = \widehat{f}_1 \circ \bl \theta$ for some function on $\bl \theta(\Omega)$ using  analogous argument as in \cite[Remark 2.2]{GN23}.
        \item[\rm 2.] Also, $R^G_{\varrho}(H^2(\Omega))=\mb P_{\varrho}(H^2(\Omega))$ \cite[Lemma 3.1]{kag}. 
    \end{enumerate}
\end{rem}
\section{The Hardy space} 
Let $\Omega$ be a bounded symmetric domain and a $G$-space for a finite complex reflection group $G.$ We define a notion of Hardy space on $\bl \theta(\Omega)$ motivated by \cite{Misra-SSR-Zhang}, $\bl\theta$ being a basic polynomial mapping associated to the group $G.$

For $\varrho \in \widehat{G}_1,$ let $c_\varrho$ denote the norm of the polynomial $\ell_\varrho$ (cf. Lemma \ref{gencz}) in $H^2(\Omega).$ By Lemma \ref{gencz} and Remark \ref{rem3}, each $g \in \mb P_{\varrho}(H^2(\Omega))$ can be written as $g=\frac{1}{c_\varrho} \ell_\varrho\cdot ( \widehat{g}\circ \bl \theta)$ for a unique holomorphic function $\widehat{g}$ on $\bl \theta(\Omega).$ 
Let $\widehat{\Gamma}_\varrho : \mb P_{\varrho}(H^2(\Omega)) \to \m O(\bl \theta(\Omega))$ be defined by 
\Bea\widehat{\Gamma}_\varrho g=\widehat{g}.
\Eea
Let $H_\varrho^2(\bl \theta(\Omega)):=\widehat{\Gamma}_\varrho(\mb P_{\varrho}(H^2(\Omega))).$ 
Since $\widehat{\Gamma}_\varrho$ is linear and injective by construction, $H_\varrho^2(\bl \theta(\Omega))$ can be made into a Hilbert space by borrowing the inner product from $H^2(\Omega),$ that is,
\Bea
 \inner{\widehat{h}}{\widehat{g}}_{H_\varrho^2(\bl \theta(\Omega))}=\langle \widehat{\Gamma}_\varrho h, \widehat{\Gamma}_\varrho g\rangle_{H_\varrho^2(\bl \theta(\Omega))}:=\langle h, g\rangle_{H^2(\Omega)} \,\, \text{for all} \,\, h, g\in\mb P_{\varrho}(H^2(\Omega)).
 \Eea
This makes the map $\widehat{\Gamma}_\varrho : \mb P_{\varrho}(H^2(\Omega)) \to H_\varrho^2(\bl \theta(\Omega)),$ a unitary. Thus, for a holomorphic function $f:\bl \theta(\Omega) \to \mb C,$ \bea\label{norm1}\norm{f}_\varrho^2:=\inner{f}{f}_{H^2_\varrho(\bl\theta(\Omega))}=\frac{1}{c^2_\varrho}\Big( {\rm sup}_{0<r<1} \int_{\del\Omega} |(f\circ \bl \theta)(rt)|^2 |\ell_\varrho(rt)|^2 d\Theta(t)\Big).\eea 
Clearly, $\norm{1}_\varrho=1.$ If $\Omega=\mb D^n,$ then for our choices of $\ell_\sgn \equiv J_{\bl\theta}$ and $\ell_{\rm tr}\equiv 1,$ we get $c_\sgn = \sqrt{|G|}$ and $c_{\rm tr} =1.$ Since for every non-zero constant $c,\,\, c \ell_\varrho$ will satisfy Lemma \ref{gencz}, so one can adjust $c_\varrho$ accordingly and consider it always to be equal to $\sqrt{|G|}$ (with an appropriate moderation in the choice of $\ell_\varrho$).

In summary, associated to each one-dimensional representation $\varrho$ of $G,$  the Hilbert space $H_\varrho^2(\bl \theta(\Omega))$ is defined as follows:
        \Bea H_\varrho^2(\bl \theta(\Omega)):=\{f:\bl \theta(\Omega) \to \mb C \text{ holomorphic and }\norm{f}_\varrho<\infty\}.\Eea The Hilbert space $H^2_\sgn(\bl \theta(\Omega))$ associated to the sign representation of $G$ is defined to be the Hardy space on $\bl \theta(\Omega)$ and is denoted  by $H^2(\bl \theta(\Omega)).$ \begin{defn} The Hardy space on $\bl \theta(\Omega)$ is defined by $$H^2(\bl \theta(\Omega)):=\{f:\bl \theta(\Omega) \to \mb C \text{ holomorphic and }\norm{f}_\sgn<\infty\}.$$
\end{defn} 

If $G$ is the permutation group $\mathfrak{S}_n$ and $\Omega=\mb D^n,$ this notion of the Hardy space coincides with the same in \cite{Misra-SSR-Zhang}. 

\subsection{Examples of the Hardy spaces on the proper images} In this subsection, we exhibit a number of examples of Hardy spaces on the proper images of the bounded symmetric domains. \begin{ex}{\sf (On the proper images of the unit polydisc)}
     \rm The \v{S}ilov boundary of the open unit polydisc $\mb D^n$ is the $n$-torus $\mb T^n,$ where $\mb T=\{z\in \mb C: |z|=1\}.$
    Let $d\Theta$ be the normalized Lebesgue measure on the torus $\mb T^n.$ 
Associated to each one-dimensional representation $\varrho$ of $G,$ the reproducing kernel Hilbert space $H_\varrho^2(\bl \theta(\mb D^n))$ is defined  as follows \cite[Section 2.2]{GG23}: 
$$H_\varrho^2(\bl \theta(\mb D^n)):= \{f\in\O(\bl\theta(\mb D^n):{\rm sup}_{0<r<1} \int_{\mb T^n} |(f\circ \bl \theta)(re^{i\Theta})|^2 |\ell_\varrho(re^{i\Theta})|^2 d\Theta < \infty\}.$$ 
This is a Hilbert space with the norm 
\Bea {\norm{f}}_\varrho =\frac{1}{c_\varrho}\Big( {\rm sup}_{0<r<1} \int_{\mb T^n} |(f\circ \bl \theta)(re^{i\Theta})|^2 |\ell_\varrho(re^{i\Theta})|^2 d\Theta\Big)^{\frac{1}{2}}.\Eea 
\begin{enumerate}[leftmargin=*]
    \item We refer to $H^2_\sgn(\bl \theta(\mb D^n))$ associated to the sign representation of $G$ as the Hardy space on $\bl\theta(\mb D^n)$ and denote it by $H^2(\bl \theta(\mb D^n)).$  \item For the sign representation of the permutation group $\mathfrak{S}_n,$ this notion of the Hardy space $H^2(\mb G_n)$ on the symmetrized polydisc coincides with the same in \cite{Misra-SSR-Zhang}.
    \item Recall from Example \ref{dk} that $D_{2k}$ acts on $\mb D^2$. The number of one-dimensional representations of the dihedral group $D_{2k}$ in $\widehat{D}_{2k}$ is $2$ if $k$ is odd and $4$ if $k$ is even. Clearly, for every $k \in \mb N$ the trivial representation of $D_{2k}$ and the sign representation of $D_{2k}$ are in $\widehat{D}_{2k}.$ Since for the trivial representation we can choose $\ell_{\rm tr}\equiv 1,$ so $c_{\rm tr}=1$ in the formula of the norm of $H_{\rm tr}^2(\mathscr{D}_{2k}).$
    \item For the sign representation, we have $\ell_{\sgn}(\bl z)= k(z_1^k -z_2^k).$ Hence $c_\sgn^2=2k^2$ in the formula of the norm of $H^2(\mathscr{D}_{2k}).$
    \item Let $k=2j$ for some $j\in \mb N.$ We consider the representation $\varrho_1$ defined by
\begin{eqnarray*} \varrho_1(\delta) = -1 &\text{~and~}& \varrho_1(\tau) =1 \text{~for~} \tau \in \inner{\delta^2}{\sigma}.
\end{eqnarray*}
It is known that (see \cite{kag}) $\ell_{\varrho_1}(\bl z) = z_1^{j} + z_2^{j}.$ Hence $c_{\varrho_1}^2=2$ in the formula of the norm of $H_{\varrho_1}^2(\mathscr{D}_{2k}).$
\item The representation $\varrho_2$ is defined as following:
\begin{eqnarray*}
\varrho_2(\delta) = -1 &\text{~and~}& \varrho_2(\tau) =1 \text{~for~} \tau \in \inner{\delta^2}{\delta\sigma}.\end{eqnarray*}
In this case, $\ell_{\varrho_2}(\bl z) = z_1^{j} - z_2^{j}$ and $c_{\varrho_2}^2=2.$
\end{enumerate}

\end{ex}

\begin{ex} {\sf (On the proper images of the unit ball)} \rm Recall that there exists $\bl \theta(\mb B_n)$ which is biholomorphic to $\mathcal E_n(m),$ cf. Example \ref{ball}. The Hardy space $H^2(\mathcal E_n(m))$ is defined as follows: 
$$H^2(\mathcal E_n(m)):=\{f\in\m O(\mathcal E_n(m)):{\rm sup}_{0<r<1} \int_{\mb S_n} m^2 |(f\circ \bl \theta)(rt)|^2 |rt_1|^{2(m-1)} d\sigma(t) < \infty\},$$ where $d\sigma$ is the normalized rotation invariant measure on the unit sphere $\mb S_n=\{(z_1,\ldots,z_n) \in \mb C^n: \sum_{i=1}^n |z_i|^2 =1\}.$ The norm of $f \in H^2(\mathcal E_n(m))$ is given  by 
\Bea{\norm{f}}_\sgn=\frac{1}{c_{m,n}} \Big({\rm sup}_{0<r<1} \int_{\mb S_n} m^2 |(f\circ \bl \theta)(rt)|^2 |rt_1|^{2(m-1)} d\sigma(t)\Big)^{\frac{1}{2}}.\Eea Since the representation is the sign representation, $c_{m,n}$ depends only on the multiplicity of the proper map $m$ and the dimension of the unit ball $n.$ For instance, $c_{m,2}=1$ and $c_{m,3}=2/(m+1)$ for every natural number $m.$ 
\end{ex}
\begin{ex}{\sf (On the tetrablock)} \rm We define Hardy space on $\mb L_3$ which is biholomorphic to the tetrablock. The domain $\mb L_3$ is a proper holomorphic image of the Lie ball $L_3,$ cf. Example \ref{lieball} and \cite{GZ2023}. The \v{S}ilov boundary of $L_3$ is given by $$\partial L_3:=\{\omega x: \omega \in \mb T \text{ and } x=(x_1,x_2,x_3)\in\mathbb R^3,x_1^2+x_2^2+x_3^2=1\}.$$ The Hardy space $H^2(\mb L_3)$ is defined as follows: 
$$H^2(\mb L_3):=\{f\in\m O(\mb L_3):{\rm sup}_{0<r<1} \int_{\partial L_3} |(f\circ \bl \theta)(rt)|^2 |rt_1|^2 d\sigma(t) < \infty\},$$ where $d\sigma$ is the normalized rotation invariant measure on $\partial L_3.$ The norm of $f \in H^2(\mb L_3)$ is given by 
\Bea{\norm{f}}_\sgn= \Big({\rm sup}_{0<r<1} \int_{\partial L_3} |(f\circ \bl \theta)(rt)|^2 |rt_1|^2 d\sigma(t)\Big)^{\frac{1}{2}}.\Eea Here one can see that $c_\sgn =2.$ \end{ex}

We note that the linear map  $\Gamma_{\varrho}^h : H^2_\varrho(\bl \theta(\Omega))  \to \mb P_{\varrho}(H^2(\Omega))$ defined by \bea\label{gamo}\Gamma_{\varrho}^h f = \frac{1}{c_\varrho} \ell_\varrho\cdot( f \circ \bl \theta),\eea
 is the adjoint of the unitary map $\widehat{\Gamma}_\varrho$ in Equation \eqref{uni}.   Due to its crucial role in the sequel, it is worth noting as the following lemma.  



 \begin{lem}\label{suriso}
 For every $\varrho \in \widehat{G}_1,$ the linear map  $\Gamma_{\varrho}^h : H^2_{\varrho}(\bl \theta(\Omega))  \to \mb P_{\varrho}(H^2(\Omega))$ is a  unitary operator. \end{lem}
In the following lemma, it is shown that $H^2_\varrho(\bl\theta(\Omega)$ is a reproducing kernel Hilbert space for every $\varrho\in\widehat G_1.$
\begin{lem}\label{repker}
For every fixed $w\in \Omega,$  there is a holomorphic function $S_{\varrho,\bl\theta}(\cdot,\bl\theta(w)) \in H^2_\varrho(\bl \theta(\Omega))$ such that $\Gamma_{\varrho}^h : H^2_\varrho(\bl \theta(\Omega))  \to \mb P_{\varrho}(H^2(\Omega))$ satisfies
$$\Gamma_{\varrho}^h : S_{\varrho,\bl\theta}(\cdot,\bl\theta(w)) \mapsto c_\varrho \frac{S_\varrho(\cdot,w)}{\overbar{\ell_\varrho(w)}},$$  where $S_\varrho$ is the reproducing kernel of $\mb P_{\varrho}(H^2(\Omega))$ and $c_\varrho = \norm{\ell_\varrho}_{H^2(\Omega)}.$   Moreover, the function $S_{\varrho,\bl\theta}: \bl\theta(\Omega)\times\bl\theta(\Omega) \to \mb C$ is the reproducing kernel for $H^2_\varrho(\bl \theta(\Omega)).$ 

\end{lem}


 \begin{proof}
  
  By the Kolmogorov decomposition of the reproducing kernel $S_\varrho$, there exists a function  $F:\Omega\ra \mathcal B(\mb P_{\varrho}(H^2(\Omega)), \C) \text{~~such that~~}$ $$S_\varrho(z,w) = F(z) F(w)^* \text{~~for~~} z, w \in \Omega$$ \cite[Theorem 2.62]{AMc}, where $\m B(X,Y)$ denotes the space of all bounded linear operators from $X$ into $Y.$  We note  that  $F(z)=\textnormal{ev}_{z}$ satisfies the requirement, where $\textnormal{ev}_{z}:f\mapsto f(z)$ is the evaluation functional at $z$. Thus, $F$ is a  holomorphic function from $\Omega$ into $ \mathcal B(\mb P_{\varrho}(H^2(\Omega)), \C)$ such that 
  $$F(z)h = h(z) \text{~~for~~} z \in \Omega \text{~~and~~} h \in \mb P_{\varrho}(H^2(\Omega)).$$ 
 
 For a fixed $w\in \Omega,$ the analytic version of the Chevalley-Shephard-Todd theorem in \cite[Theorem 3.2, Theorem 3.12]{BDGS22} yields the following representation of the kernel function $$S_\varrho(z,w)= \ell_\varrho(z) \widehat{S}_{\varrho,\bl\theta}(\bl \theta(z),w),$$ where $\widehat{S}_{\varrho,\bl\theta}(\bl \theta(z),w)$ is a unique  $G$-invariant holomorphic function in $z$ and is anti-holomorphic function in $w.$ 
 
Let $G=\{\alpha_i:i=1,\ldots,d\}$ and $\{p_1,\ldots,p_d\}$ be a basis of the module $\mb C[z_1,\ldots,z_n]$ over the ring $\mb C[\theta_1,\ldots,\theta_n].$ Without loss of generality, we assume that $p_1=\ell_\varrho$  and invoking \cite[Lemma 3.11]{BDGS22} write the following expression of $\widehat{S}_{\varrho,\bl\theta}(\bl \theta(z),w):$
$$\widehat{S}_{\varrho,\bl\theta}(\bl \theta(z),w) = \frac{\det \big( \Lambda^{(1)}_1 F(z) \big)}{\det \Lambda(z)} F(w)^*,$$ where $\Lambda(z)=\big(\!\!\big(\big(\alpha_i\big(p_j(z)\big)\big)\!\!\big)_{i,j=1}^d$ and $ \Lambda^{(1)}_1F(z)$ is the matrix $\Lambda(z)$ with its first column replaced by the column $\Big(\big(\!\!\big(\alpha_i(F(z))\big)\!\!\big)_{i=1}^d\Big)^{\rm tr}.$ This implies that $F_1(z)=\frac{\det \big( \Lambda^{(1)}_1 F(z) \big)}{\det \Lambda(z)}$ is in $\mathcal B(\mb P_{\varrho}(H^2(\Omega)\!), \C).$ Hence $F_1(z)^* \in \mathcal B(\C,\mb P_{\varrho}(H^2(\Omega)\!))$ and so there exists $h\in \mb P_{\varrho}(H^2(\Omega))$ satisfying $F_1(z)^*1=h.$ Thus, $$\ov{\widehat{S}_{\varrho,\bl\theta}(\bl \theta(z),w)}= F(w)F_1(z)^*1=h(w).$$  So, for a fixed $z,$ the function  $w\mapsto \ov{\widehat{S}_{\varrho,\bl\theta}(\bl \theta(z),w)}$ is in $\mb P_{\varrho}(H^2(\Omega)).$   Now another application of \cite[Theorem 3.2, Theorem 3.12]{BDGS22} to $\ov{\widehat{S}_{\varrho,\bl\theta}(\bl \theta(z),w)}$ (as a function of $w$) yields:
\bea \label{kerfor} S_\varrho(z,w)= \ell_\varrho(z) \widetilde{S}_{\varrho,\bl\theta}(\bl \theta(z),\bl \theta(w)) \ov{\ell_\varrho(w)},\eea where $\widetilde{S}_{\varrho,\bl\theta}(\bl \theta(z),\bl \theta(w))$ is unique and holomorphic function in $z,$ anti-holomorphic in $w.$  

Let $S_{\varrho,\bl\theta}(\bl \theta(z),\bl \theta(w)):=c_\varrho^2\widetilde{S}_{\varrho,\bl\theta}(\bl \theta(z),\bl \theta(w)).$ For a fixed $w,$ it follows from the definition that $\norm{S_{\varrho,\bl\theta}(\cdot,\bl \theta(w))}_\varrho<\infty.$ We complete the proof by showing the reproducing property of $S_{\varrho,\bl\theta}:\bl\theta(\Omega)\times\bl\theta(\Omega)\to\mb C.$  If $f\in H^2_\varrho(\bl \theta(\Omega)),$ then \Bea \inner{f}{ S_{\varrho,\bl\theta}(\cdot,\bl \theta(w))} &=& \inner{\Gamma_\varrho^h f}{\Gamma_\varrho^h S_{\varrho,\bl\theta}(\cdot,\bl \theta(w))}\\&=&  \inner{\frac{1}{c_\varrho}\ell_\varrho (f\circ \bl \theta)}{c_\varrho\ell_\varrho \widetilde{S}_{\varrho,\bl\theta}(\bl\theta(\cdot),\bl \theta(w))}\\&=& \inner{\ell_\varrho (f\circ \bl \theta)}{\frac {S_{\varrho}(\cdot,w)}{\ov{\ell_\varrho(w)}}}  \\&=&  f(\bl \theta(w)),  \Eea 
where the last equality follows from Equation \eqref{kerfor} and the reproducing property of $S_\varrho(\cdot,w).$ 
\end{proof}

Combining Lemma \ref{suriso} and Lemma \ref{repker}, we conclude the following result. 
\begin{prop}\label{idenh} 
For every $\varrho\in\widehat{G}_1,$ the reproducing kernel Hilbert space $H^2_\varrho(\bl \theta(\Omega))$ is isometrically isomorphic to $\mb P_{\varrho}(H^2(\Omega)).$\end{prop}
\begin{rem}\label{divisible}
For every fixed $z\in\Omega,$ the function $w\mapsto\ov{S_\varrho(z,w)}$ is in $\mb P_\varrho(H^2(\Omega)).$ So 
$$\ov{S_\varrho(z,w)}=\ell_\varrho(w) (f_z\circ\bl\theta)(w)$$ 
for some unique $G$-invariant holomorphic function $f_z\circ \bl\theta$ on $\Omega.$ By uniqueness in \cite[Theorem 3.2, Theorem 3.12]{BDGS22}, it follows that 
    $$\ov{(f_z\circ\bl\theta)(w)}=\ell_\varrho(z) \tilde{S}_{\varrho,\bl\theta}(\bl \theta(z),\bl \theta(w)).$$ Moreover, for a fixed $w\in\Omega,$ (even if $\ov{\ell_\varrho(w)}=0$) the map $z\mapsto \ov{(f_z\circ\bl\theta)(w)}=\frac{S_\varrho(z,w)}{\ov{\ell_\varrho(w)}}$ is well-defined and holomorphic in $z\in \Omega.$ For instance, if $\Omega=\mb D^2$ and $G=\mathfrak{S}_2,$ then  $\ell_\sgn(w)=0$ at  $w=(0,0),$ whereas $$\frac{S_\sgn(z,w)}{\ov{\ell_\sgn(w)}}=\ell_\sgn(z).$$\end{rem} 
From Equation \eqref{repvar} and \eqref{kerfor}, it follows that \bea\label{kernel}  S_{\varrho,\bl \theta} (\bl \theta(z), \bl \theta(w))  =\frac{c_\varrho^2}{|G|}\frac{1}{ \ell_\varrho(z) \ov{\ell_\varrho(w)}}\sum_{\sigma \in G} \chi_{\varrho}(\sigma^{-1}) S_\Omega(\sigma^{-1} \cdot z, w),\eea 
where $S_\Omega$ is the reproducing kernel of $H^2(\Omega).$ The reproducing kernel $S_{\sgn,\bl\theta}$ of  $H^2(\bl\theta(\Omega))$ is called the Szeg\"{o} kernel of $\bl\theta(\Omega).$ Explicit formulae for the Szeg\"{o} kernels for different choices of $\Omega$ and basic polynomial maps $\bl\theta$ can obtained by appealing to Equation \eqref{kernel}  in a manner analogous to that of \cite{kag} for the case of weighted Bergman kernels. A few examples are derived in Subsection \ref{exsze}.

\begin{rem}\rm
  We would like to point out that the definition of $H^2(\bl \theta(\Omega))$ is independent of the choice of the basic polynomial map $\bl\theta$ associated  to $G.$ 
  \begin{itemize}[leftmargin=*] \item   Let $\bl\theta^\prime:\Omega \to \bl\theta^\prime(\Omega)$ be another basic polynomial mapping associated to the group $G.$ Since  there is a biholomorphic map $\bl h : \bl \theta(\Omega) \to \bl \theta^\prime(\Omega),$ that is, $\bl h \circ \bl \theta=\bl \theta^\prime,$ it follows from the chain rule and Corollary \ref{Jac}  that $J_{\bl h}(\bl \theta(z)) =c$ for all $z\in \Omega,$ where $c$ is some non-zero constant. The linear map $U:H^2(\bl \theta^\prime(\Omega)) \to H^2(\bl \theta(\Omega))$ defined by $U(f)=c\cdot (f\circ \bl h)$ is a unitary. In fact, 
  $$S_{{\rm sgn},\bl \theta} (\bl \theta(z), \bl \theta(w))= \vert c\vert^2 S_{{\rm sgn},\bl \theta^\prime} (\bl \theta^\prime(z), \bl \theta^\prime(w)) \text{~~for~~} z, w\in \Omega.$$
  Therefore, $H^2(\bl \theta^\prime(\Omega))$ and $H^2(\bl \theta(\Omega))$ are isometrically isomorphic to each other.
  
  \item Let $\varrho \in \widehat{G}_1$ be a representation that is not isomorphic to the sign representation. Following analogous arguments as above, one can show that the definition of $H^2_\varrho(\bl \theta(\Omega))$ is independent of the choice of $\bl \theta.$   By the analytic Chevalley-Shephard-Todd theorem\cite[Theorem 3.12]{BDGS22}, every element $f\in\mb P_{\varrho}(H^2(\Omega))$ can be expressed as 
  $$f=\ell_\varrho\cdot (g\circ\bl \theta^\prime) = \ell_\varrho\cdot( g\circ\bl h \circ \bl \theta).$$ 
  We note from Proposition \ref{idenh}  that $g\circ\bl h \in H^2_\varrho(\bl \theta(\Omega))$ and $g\in H^2_\varrho(\bl \theta^\prime(\Omega)).$  Since the  map $U_\varrho: H^2_\varrho(\bl \theta^\prime(\Omega)) \to H^2_\varrho(\bl \theta(\Omega))$ defined by $U_\varrho(g) = g\circ \bl h$ is a unitary, $H^2_\varrho(\bl \theta^\prime(\Omega))$ and $ H^2_\varrho(\bl \theta(\Omega))$ are isomorphically isometric.  In other words, $U_\varrho=\Gamma_2^*\Gamma_1,$ where $\Gamma_1 : H^2_\varrho(\bl \theta^\prime(\Omega)) \to \mb P_{\varrho}(H^2(\Omega))$ and $\Gamma_2 : H^2_\varrho(\bl \theta(\Omega)) \to \mb P_{\varrho}(H^2(\Omega))$ are the unitary operators in Lemma \ref{suriso}. Moreover,
  $$S_{\varrho,\bl \theta} (\bl \theta(z), \bl \theta(w))= S_{\varrho,\bl \theta^\prime} (\bl \theta^\prime(z), \bl \theta^\prime(w)) \text{~~for~~} z, w\in \Omega.$$
  \end{itemize} To eliminate any ambiguity in the two points mentioned above, we note that since we always choose $\ell_{\rm sgn} =J_{\bl \theta},$ it follows that $\norm{1}_{H^2(\bl \theta^\prime(\Omega))}=c \norm{1}_{H^2(\bl \theta(\Omega))}.$ So we had to adjust the operator $U$ with a constant to make it an isometry. However, for any other one-dimensional representation $\varrho$, we do not choose different $\ell_\varrho$'s for $H^2_\varrho(\bl \theta^\prime(\Omega))$ and $ H^2_\varrho(\bl \theta(\Omega))$, so in the description of $U_\varrho$ no adjustment is needed.
  
\end{rem}

\subsection{Formula for the Szeg\"o Kernel}\label{exsze} Let $\bl\theta=(\theta_1,
\ldots,\theta_n)$ be a basic polynomial for $G(m,p,n)$ as described in Example \ref{exo}. It is easy to see that 
$$ J_{\bl\theta}(z)=\frac{m^n}{p}(z_1z_2\cdots z_n)^{\frac{m}{p}-1}\prod_{i<j}(z_i^m-z_j^m)\text{~and~}c_\sgn=\Vert J_{\bl\theta}\Vert=\frac{m^n\sqrt{n!}}{p}.$$
Choosing $\varrho=\sgn, \,\Omega=\mb D^n$ in Equation \eqref{kernel} and recalling that $\vert G(m,p,n)\vert=\frac{m^n n!}{p}$, it follows that the Szeg\"{o} kernel for $\bl\theta(\mb D^n)$ is given by 
\Bea
&&S_{\bl\theta(\mb D^n)}\big(\bl\theta(z),\bl\theta(w)\big)\\&=&S_{\sgn,\bl\theta}\big(\bl\theta(z),\bl\theta(w)\big)\\&=&\frac{p}{m^n}\frac{s_n(z)\ov{s_n(w)}}{\theta_n(z)\ov{\theta_n(w)}\displaystyle\prod_{i<j}(z_i^m-z_j^m)(\bar w_i^m-\bar w_j^m)}\displaystyle\sum_{\sigma\in G(m,p,n)}\chi_\sgn(\sigma^{-1})S_{\mb D^n}(\sigma^{-1}\cdot z,w),
\Eea  where $S_{\mb D^n}(z,w)=\displaystyle\prod_{j=1}^n(1-z_j\bar w_j)^{-1}.$ 
\begin{enumerate}[leftmargin=*]
\item [1.] The dihedral group $D_{2k}=G(k, k,2)$ acts on $\mb D^2$ (cf. Example \ref{dk}) and 
$$\theta_1(z)=z_1^k+z_2^k,\,\,\theta_2(z)=z_1z_2,\,
J_{\bl\theta}(z)=k(z_1^k-z_2^k).$$ Recall that $\bl\theta(\mb D^2)=\mathscr D_{2k}.$ The reproducing kernel for $H^2(\mathscr D_{2k})$ is given by 

\Bea
S_{\mathscr D_{2k}}\big(\bl\theta(z),\bl\theta(w)\big)=\frac{1}{k(z_1^k-z_2^k)(\bar w_1^k-\bar w_2^k)}\sum_{\sigma\in D_{2k}}\chi_\sgn(\sigma^{-1})S_{\mb D^2}(\sigma^{-1}\cdot z, w).
\Eea

\item[2.] The group $\mathfrak{S}_n=G(1, 1,n),\,n>1$ acts on $\mb D^n$ (cf. Example \ref{symmet}) and the symmetrization map $$\bl s = (s_1,\ldots,s_n): \mb D^n \to \mb G_n$$
is a basic polynomial associated to $\mathfrak{S}_n,$ where $s_k$'s  are elementary symmetric polynomials of degree $k$ in $n$ variables, defined in Equation \eqref{sym}.  Noting that  $J_{\bl s}(z)=\prod_{i<j}(z_i-z_j)$ \cite[Lemma 10]{EZ}, it follows that the Szeg\"{o} kernel for $\mb G_n$ is given by 
 
\Bea
S_{\mb G_n}\big(\bl s(z),\bl s(w)\big)&=& \frac{1}{\displaystyle\prod_{i<j}(z_i-z_j)(\bar w_i-\bar w_j)}\sum_{\sigma\in \mathfrak S_n}\chi_\sgn(\sigma^{-1})S_{\mb D^n}(\sigma^{-1}\cdot z, w)\\
&=& \frac{1}{\displaystyle\prod_{i<j}(z_i-z_j)(\bar w_i-\bar w_j)}\sum_{\sigma\in \mathfrak S_n}\sgn(\sigma)\displaystyle\prod_{j=1}^n(1-z_j \bar w_{\sigma(j)})^{-1}\\
&=& \frac{1}{\displaystyle\prod_{i<j}(z_i-z_j)(\bar w_i-\bar w_j)}\det\bigg(\big(\!\!\big((1-z_i\bar w_j)^{-1}\big)\!\!\big)_{i,j=1}^n\bigg)\\
&=&\displaystyle\prod_{i.j=1}^n(1-z_i\bar w_j)^{-1},
\Eea  where the last equality follows from \cite[p.2367]{Misra-SSR-Zhang}.

\item[3.] Let $\Lambda:\m R_{III}(2)\to \mb C^3,\,\Lambda(z):=(z_1,z_2,z_1z_2-z_3^2),$ where $\m R_{III}(2)$ is as described in Example \ref{car2} and we identify $z=(z_1,z_2,z_3)\in\mb C^3$ with a $2\times 2$ symmetric matrix 
$\begin{bmatrix}
  z_1 & z_3 \\
  z_3 & z_2
\end{bmatrix}.$ Then $\Lambda$ is a proper holomorphic map of multiplicity $2$ which is factored by the group $\mb Z_2.$ The domain $\Lambda(\m R_{III}(2):=\mb E,$ is called the tetrablock. 

The Szeg\"{o} kernel of $\m R_{III}(2)$ is given by
$$S_{\m R_{III}(2)}(z,w)=\bigg[\det\bigg(\begin{bmatrix}
  1 & 0 \\
  0 & 1
\end{bmatrix}-\begin{bmatrix}
  z_1 & z_3 \\
  z_3 & z_2
\end{bmatrix}\begin{bmatrix}
  \bar w_1 & \bar w_3 \\
  \bar w_3 & \bar w_2
\end{bmatrix}\bigg)\bigg]^{-3/2}$$ for $z=(z_1,z_2,z_3)$ and $w=(w_1,w_2,w_3)\in\m R_{III}(2)$ \cite[p. 29]{Arazy}.
It is easy to see that $J_\Lambda(z)=-2z_3.$ The Szeg\"{o} kernel for $\mb E$ is given by 

$$S_{\mb E}\big(\Lambda(z),\Lambda(w)\big)=\frac{S_{\m R_{III}(2)}(z,w)-S_{\m R_{III}(2)}(\sigma^{-1}\cdot z,w)}{4z_3\bar w_3}$$ for  $z=(z_1,z_2,z_3),$ $\sigma^{-1}\cdot z=(z_1,z_2,-z_3)$ and $w=(w_1,w_2,w_3)\in\m R_{III}(2).$ \end{enumerate}
\begin{rem}
Note that \Bea J_{\bl\theta}(z)&=&(z_1\ldots z_n)^{\frac{m}{p}-1}\det\Big(\big(\!\!\big(z_j^{m(n-i)}\big)\!\!\big)_{i,j=1}^n\Big)\\
&=&\det\bigg(\Big(\!\!\Big(z_j^{\frac{m}{p}(p(n-i)+1)-1}\Big)\!\!\Big)_{i,j=1}^n\bigg).
\Eea
$$\Vert\det\bigg(\Big(\!\!\Big(z_j^{\frac{m}{p}(p(n-i)+1)-1}\Big)\!\!\Big)_{i,j=1}^n\bigg)\Vert^2=n!\prod_{k=1}^n\frac{(\frac{m}{p}(p(n-k)+1)-1)!}{(\lambda)_{\frac{m}{p}(p(n-k)+1)-1}}.$$
\end{rem}
\subsection{Orthonormal basis} We obtain an orthonormal basis of $H^2_\varrho(\bl \theta(\Omega))$ applying Proposition \ref{idenh}. Let $\{e_{\bl \alpha}:\bl \alpha \in \m I\}$ be an orthonormal basis for $H^2(\Omega)$  \cite{Hua1963}. 
\begin{itemize}[leftmargin=*]
\item Suppose that $\sigma \cdot \bl \alpha \in \m I$ for every $\sigma \in G,$ also,  $e_{\sigma \cdot\bl \alpha}$ and $e_{\tau \cdot\bl \alpha}$ are mutually orthogonal  whenever $\sigma \neq \tau.$ Since for every $\varrho \in \widehat{G}_1,$ $\sum_{\sigma\in G} |\chi_\varrho(\sigma)|^2=|G|,$ it follows that 
$\norm{\mb P_\varrho e_{\bl \alpha}}=\frac{1}{\sqrt{|G|}}.$ Moreover, if $\bl \beta \neq \sigma \cdot \bl \alpha$ for all $\sigma \in G,$ then $\mb P_\varrho e_{\bl \alpha}$ and $\mb P_\varrho e_{\bl \beta}$ are orthogonal to each other. 
\item If $e_{\bl \alpha}$'s are monomials and $\bl \beta = \sigma \cdot \bl \alpha$, then  $\mb P_\varrho e_{\bl \beta}=\chi_\varrho(\sigma) \mb P_\varrho e_{\bl \alpha}.$  In fact, if $\varrho\in\widehat G$ is not equivalent to the trivial representation, there exists at least one $\sigma_0\in G$ which satisfies $\chi_\varrho(\sigma_0)\neq 1.$ Let $\bl\alpha\in\m I$ be such that  $\bl \alpha=\sigma_0 \cdot \bl \alpha,$  then $\mb P_\varrho e_{\bl \alpha}=0.$

\end{itemize}
Let $\tilde{\m I}_\varrho := \{\bl \alpha \in \m I : \mb P_\varrho e_{\bl \alpha} \neq 0 \}.$ Choosing  elements from $\{\mb P_\varrho e_{\bl \alpha}:\bl \alpha \in \tilde{\m I}_\varrho\},$ we obtain an orthogonal basis of $\mb P_\varrho(H^2(\Omega)).$ We describe a scheme to make such a choice in the following examples.      

\begin{ex}
    Suppose the domain $\Omega$ is either the open unit polydisc $\mb D^n$ or the unit ball $\mb B_n$ in $\mb C^n.$ Let $\mb N_0$ be the set of all non-negative integers. For $\bl m = (m_1,\ldots,m_n) \in \mb N_0^n,$ $z^{\bl m} = \prod_{i=1}^n z_i^{m_i},$ $ z=(z_1,\ldots,z_n) \in \mb C^n.$ Note that $\{k_{\bl m}z^{\bl m} : \bl m \in \mb N_0^n \}$ forms an orthonormal basis of $H^2(\Omega),$ where $k_{\bl m}=1$ for $\mb D^n$ and $k_{\bl m}=\sqrt{\frac{(n-1+\sum_i m_i)!}{m_1!\cdots m_n!(n-1)!}}$ for $\mb B_n.$ For $\varrho \in \widehat{G}_1,$ let 
    $$\tilde{\m I}_\varrho = \{\bl m \in \mb N_0^n : \mb P_\varrho z^{\bl m} \neq 0 \} \text{~~and~~} S_{\varrho, G} = \{\sigma \in G: \chi_\varrho(\sigma)=1\}.$$
    For some $\sigma \in S_{\varrho, G}$  and $\bl m \in \tilde{\m I}_\varrho$ such that $\sigma \cdot \bl m \neq \bl m,$ we have $\mb P_\varrho z^{\sigma \cdot \bl m}=\mb P_\varrho z^{\bl m}.$ Let 
    $$[\bl m]=\{\sigma \cdot \bl m : \sigma \cdot \bl m \neq \bl m \in \tilde{\m I}_\varrho \text{ for } \sigma \in S_{\varrho, G} \} \text{~~and~~} \m I_\varrho=\{ [\bl m]:\bl m \in \tilde{\m I}_\varrho\}.$$
    Let
    $$e_{\bl m}(\bl \theta(z)) := c_\varrho \frac{\mb P_\varrho (k_{\bl m}z^{\bl m})}{\ell_\varrho(z)} \text{~~for~~} [\bl m] \in \m I_\varrho.$$ 
    It follows from Proposition \ref{idenh}  that $\{\sqrt{|G|} e_{\bl m}: [\bl m] \in \m I_\varrho\}$ is an orthonormal basis of $H_\varrho^2(\bl \theta(\Omega)).$
\end{ex} 

The index set $\m I_\varrho$ can be determined explicitly in particular cases. Suppose that $\Omega=\mb D^n$ and $G=\mathfrak{S}_n,$ then $\bl\theta(\mb D^n)$ is biholomorphic to $\mb G_n$ cf. Example \ref{symmet}. The trivial representation (tr) and the sign representation (sgn) are the only one-dimensional representations of the permutation group $\mathfrak{S}_n$. 
\begin{itemize}[leftmargin=*]
\item As per our choice of $\ell_\sgn=J_{\bl s}$ and $\ell_{\rm tr}=1,$ one gets $c_\sgn=\sqrt{n!}$ and $c_{\rm tr}=1.$ \item $\m I_\sgn = \{\bl m \in \mb N_0^n : 0\leq m_1<m_2<\cdots<m_n \}$  and for each $\bl m \in \m I_\sgn,$ 
$$ e_{\bl m}(\bl s(z)) = \frac{1}{\sqrt{n!}} \frac{\bl a_{\bl m}(z)}{\prod_{i<j}(z_i-z_j)}, \text{~~where~~} \bl a_{\bl m}(z) = \det\Big((\!(z_i^{m_j})\!)_{i,j=1}^n\Big),$$
and $\bl s$ is the symmetrization map in Equation \eqref{sym}. The set $\{\sqrt{n!}e_{\bl m} : \bl m \in \m I_\sgn \}$ forms an orthonormal basis for $H^2(\mb G_n)$ \cite{Misra-SSR-Zhang}.

\item Also, $\m I_{\rm tr} = \{\bl m \in \mb N_0^n : 0\leq m_1\leq m_2\leq\cdots\leq m_n \}$  and  $ f_{\bl m}(\bl s(z)) = \frac{1}{n!} \bl p_{\bl m}(z) $ for  $\bl m \in \m I_{\rm tr},$ where $\bl p_{\bl m}(z) = {\rm perm}\Big((\!(z_i^{m_j})\!)_{i,j=1}^n\Big),$ here {\rm perm}$A$ denotes the  the permanent of the matrix $A.$ The set $\{\sqrt{n!}f_{\bl m} : \bl m \in \m I_{\rm tr} \}$ forms an orthonormal basis for $H^2_{\rm tr}(\mb G_n)$. 
\end{itemize}

\subsection{Analytic Hilbert Module} 
Suppose $\Omega\subset \mb C^n$ is a bounded symmetric domain. It is a well-known fact that $H^2(\Omega)$ is an analytic Hilbert module over  $\mb C[z_1,\ldots,z_n].$ We observe that this rich structure can be transferred to $H^2(\bl\theta(\Omega))$ in a proper setting.

We first recall two definitions from \cite{DP}.   A Hilbert space $\mathcal H$ is said to be a {\emph Hilbert module} over an algebra $\m A$ if the map $$(f,h) \mapsto f\cdot h,\,\,\,\,f\in \m A, h\in \mathcal H,$$ defines an algebra homomorphism $f \mapsto T_f$ of $\m A$ into $\m B(\m H),$ where $\m B(\m H)$ is the algebra of bounded operators on $\m H$ and $T_f$ is the bounded operator defined by  $T_f h = f \cdot h.$ 
\begin{defn}
A Hilbert module $\m H$ (consisting of complex-valued holomorphic functions on $\Omega\subseteq\mb C^n$) over the polynomial algebra $\mb C[z_1,\ldots,z_n]$ is said to be an {\it analytic Hilbert module}  if
\begin{enumerate}
\item [(1)] $\mb C[z_1,\ldots,z_n]$  is dense in $\m H$ and
\item[(2)] $\m H$ possesses a reproducing kernel on $\Omega.$
\end{enumerate}
The module action in an analytic Hilbert module is given by pointwise multiplication, that is, for every $p\in \mb C[z_1,\ldots,z_n]$ the module action is $${\mathfrak m}_p(h)(z) = p(z) {h}(z),\, h\in \m H \text{ and } z\in \Omega.$$ 
\end{defn}

 Since $H^2(\Omega)$ is an analytic Hilbert module over the polynomial algebra, each multiplication operator $M_{\theta_i} : H^2(\Omega) \to H^2(\Omega)$ is bounded for $i=1,\ldots,n.$ Moreover, $\mb P_\varrho (H^2(\Omega))$ is a Hilbert module over the polynomial algebra $\mb C[\theta_1,\ldots,\theta_n]$ for every $\varrho\in\widehat G_1.$ 

For $i=1,\ldots,n$ and $\varrho \in \widehat{G}_1,$ let $M_i:H_\varrho^2(\bl \theta(\Omega))\to H_\varrho^2(\bl \theta(\Omega))$ be the $i$-th coordinate multiplication operator. The unitary $\Gamma_\varrho^h $ defined in Equation \eqref{gamo} intertwines $(M_1,\ldots,M_n)$ on $H_\varrho^2(\bl \theta(\Omega))$ and $(M_{\theta_1},\ldots,M_{\theta_n})$ on $\mb P_\varrho(H^2(\Omega)).$

Further, $\mb P_\varrho (\mb C[z_1,\ldots,z_n])= \ell_\varrho \cdot \mb C[\theta_1,\ldots,\theta_n]$ \cite{Pan} and is dense in $\mb P_\varrho( H^2(\Omega)).$  This implies that $\mb C[z_1,\ldots,z_n]$  is dense in $H_\varrho^2(\bl \theta(\Omega))$ and leads to the following result. 
\begin{prop}\label{idenh2}
    For every $\varrho \in \widehat{G}_1,$ the reproducing kernel Hilbert space $H_\varrho^2(\bl \theta(\Omega))$ is an analytic Hilbert module on $\bl \theta(\Omega)$ over $\mb C[z_1,\ldots,z_n].$ Moreover, $H_\varrho^2(\bl \theta(\Omega))$ over $\mb C[z_1,\ldots,z_n]$ is unitraily equivalent to the Hilbert module $\mb P_\varrho(H^2(\Omega))$ over $\mb C[\theta_1,\ldots,\theta_n].$
\end{prop} 

\subsection{Equivalence of spaces}For every $\varrho \in \widehat{G}_1,$ let $d\Theta_{\varrho}$ be the measure supported on the \v{S}ilov boundary $\del\bl \theta(\Omega)$ obtained from the following equality:
\bea\label{measure} \int_{\del\bl \theta(\Omega)} f d\Theta_{\varrho} = \int_{\del\Omega} (f\circ \bl \theta )|\ell_\varrho|^2 d\Theta ,\eea where $\ell_\varrho$ is as defined in Lemma \ref{gencz}. The $L^2$-space on $\del\bl \theta(\Omega)$ with respect to the measure $d\Theta_{\varrho}$ is given by \Bea L^2_\varrho(\del\bl \theta(\Omega))= \{f : \del\bl \theta(\Omega) \to \mb C ~\mbox{measurable}~ | \int_{\del\bl \theta(\Omega)} |f|^2 d\Theta_{\varrho} < \infty \}.\Eea
In the next couple of lemmas we realize $H^2_\varrho(\bl \theta(\Omega))$ as a closed subspace of $L^2_\varrho(\del\bl \theta(\Omega)).$

\begin{lem}\label{iden}
For every one-dimensional representation $\varrho $ of $G,$ the space $L^2_\varrho(\del\bl \theta(\Omega))$ is isometrically isomorphic to $\mb P_{\varrho}(L^2(\del\Omega)).$
\end{lem}
\begin{proof}
 The linear map $\Gamma_{\varrho}^\ell : L^2_\varrho(\del\bl \theta(\Omega))  \to \mb P_{\varrho}(L^2(\del\Omega))$ defined by \bea\label{gam}\Gamma_{\varrho}^\ell f = \frac{1}{c_\varrho} \ell_\varrho\cdot (f \circ \bl \theta)\eea is an isometry.  For  $\phi\in\mb P_{\varrho}(L^2(\del\Omega)),$  Remark \ref{rem3} guarantees the  existence of $\widehat{\phi}$ satisfying $\phi = \frac{1}{c_\varrho}\ell_\varrho \cdot(\widehat{\phi} \circ \bl \theta).$   Clearly, $\Vert \widehat \phi\Vert=\Vert \phi\Vert$ and $\widehat{\phi} \in L^2_{\varrho}(\del\bl \theta(\Omega)).$  Hence $\Gamma_{\varrho}^\ell$ is a unitary.  
\end{proof}

\begin{lem}\label{equi}
For every one-dimensional representation $\varrho $ of $G,$ $H^2_\varrho(\bl \theta(\Omega))$ is isometrically embedded in $L^2_\varrho(\del\bl \theta(\Omega)).$ 
\end{lem}
\begin{proof}
There is an isometric isomorphism of $H^2(\Omega)$ onto a closed subspace of $L^2(\del \Omega)$ \cite{Koranyi1965}, \cite[p. 526, Theorem 6]{Hahn-Mitchell:1969}. More precisely, from \cite[p. 526, Theorem 6]{Hahn-Mitchell:1969}, it is clear that the isometric isomorphism sends a function of $H^2(\Omega)$ to its radial limit. Moreover, the  discussion in Subsection 2.4 implies that if $f \in \mb P_\varrho(H^2(\Omega))$ then its radial limit function is in $\mb P_\varrho(L^2(\del \Omega)).$

If $\widehat{i}_\varrho : \mb P_\varrho(H^2(\Omega)) \to \mb P_\varrho(L^2(\del\Omega))$ denotes the isometric embedding, then it follows that the following diagram commutes:
\[ \begin{tikzcd}
H^2_\varrho(\bl \theta(\Omega))\arrow{r}{{\Gamma^{\ell*}}_\varrho \circ \widehat{i}_\varrho\circ\Gamma_\varrho^h} \arrow[swap]{d}{\Gamma^h_\varrho} & L^2_{\varrho}(\del\bl \theta(\Omega))  \arrow{d}{\Gamma^\ell_\varrho} \\%
\mb P_\varrho(H^2(\Omega)) \arrow{r}{\widehat{i}_\varrho}& \mb P_\varrho(L^2(\del\Omega))
\end{tikzcd}
\]
Thus,  the isometry ${\Gamma^{\ell*}}_\varrho \circ \widehat{i}_\varrho\circ\Gamma_\varrho^h$ is an  embedding of  $H^2_\varrho(\bl \theta(\Omega))$ into $L^2_\varrho(\del\bl \theta(\Omega)).$
\end{proof}
Equivalently, there is a closed subspace of $L^2_\varrho(\del\bl \theta(\Omega))$ which is isometrically isomorphic to $H^2_\varrho(\bl \theta(\Omega)).$ 

\subsection{Essentially bounded functions}\label{esssup} 
For each one-dimensional representation $\varrho$ of $G,$ we define $L_\varrho^\infty(\del\bl \theta(\Omega)):=\{f:\del\bl \theta(\Omega) \to \mb C \text{ measurable, essentially bounded  w.r.t } d\Theta_{\varrho}\}$ and $L^\infty(\del\Omega)^G:=\{f\in L^\infty(\del\Omega): \text{for all~}\sigma\in G,~ \sigma(f)=f \text{~a.e.}\}.$
  The map $i_\varrho:u\mapsto u\circ\bl\theta$ is an isometric $*$-isomorphism of $L_\varrho^\infty(\del\bl \theta(\Omega))$ onto $L^\infty(\del\Omega)^G.$ Indeed, each $i_\varrho$ is well-defined, since $\del\bl \theta(\Omega)=\bl \theta(\del\Omega).$  For $u \in L_\varrho^\infty(\del\bl \theta(\Omega)),$ the multiplication operator $M_u$ on $L^2_{\varrho}(\del\bl \theta(\Omega)$ is bounded   and the algebra $*$-isomorphism $i:u\mapsto M_u$ of   $L_\varrho^\infty(\del\bl \theta(\Omega))$ into $\m B(L_\varrho^2(\del\bl \theta(\Omega)))$  is isometric.  Thus,  the following  diagram commutes:
\[ \begin{tikzcd}
L_\varrho^\infty(\del\bl \theta(\Omega)) \arrow{r}{i_\varrho} \arrow[swap]{d}{i} & L^\infty(\del\Omega)^G \arrow{d}{\widetilde{i}} \\
\m B(L^2_{\varrho}(\del\bl \theta(\Omega)) \arrow{r}{j_\varrho}& \m B(\mb P_\varrho(L^2(\del\Omega)))
\end{tikzcd}
\]
where $j_\varrho(X) = \Gamma_\varrho^\ell X \Gamma_\varrho^{\ell*}$ and $\widetilde{i}(\widetilde{u})=M_{\widetilde{u}}$ denote natural inclusion maps.
It evidently follows, since for every $u$ in $L_\varrho^\infty(\del\bl \theta(\Omega))$ and $f \in \mb P_\varrho(L^2(\del\Omega)),$ one has 
$$ \Gamma_\varrho^\ell M_u \Gamma_\varrho^{\ell*}f = (u\circ \bl \theta)~ \Gamma_\varrho^\ell \Gamma_\varrho^{\ell*}f =M_{u\circ \bl \theta} f.$$ 
Let \bea\label{linfnty} L^\infty(\del\bl \theta(\Omega)):=\{u:\del\bl\theta(\Omega)\to\mb C ~~\text{measurable}:u\circ\bl\theta\in L^\infty(\del\Omega)^G \}.\eea   If $u \in L^\infty(\del\bl \theta(\Omega)),$ then $u\in L_\varrho^\infty(\del\bl \theta(\Omega))$ for every $\varrho \in \widehat{G}_1$ and conversely. For $ u \in L^\infty(\del\bl \theta(\Omega)),$ the Laurent operator $M_u$ on $L^2_\varrho(\del\bl \theta(\Omega))$ is defined  by \bea\label{laurent}M_uf=u f .\eea 
The above discussion is summarized in the following lemma.
 
 \begin{lem}
If $u\in L^\infty(\del\bl \theta(\Omega)),$  then $M_u$ on 
$L^2_{\varrho}(\del\bl \theta(\Omega))$ is unitarily equivalent to $M_{\widetilde{u}}$ on $\mb P_\varrho(L^2(\del\Omega))$ for every $\varrho\in\widehat G_1,$ where $\widetilde{u}=u\circ\bl\theta\in L^\infty(\del\Omega)^G.$ 
\end{lem} 
\section{Toeplitz operators}
 We start this section  with the definition of  Toeplitz operator on $H^2_\varrho(\bl \theta(\Omega)).$ Let $P_\varrho: L^2_\varrho(\del\bl \theta(\Omega))\to H^2_\varrho(\bl \theta(\Omega)) $ be the orthogonal projection. 
\begin{defn}
For $u\in L^\infty(\del\bl \theta(\Omega)),$ the Toeplitz operator $T_u$ is defined on $H^2_\varrho(\bl \theta(\Omega))$  by 
\bea\label{to} T_u = P_\varrho M_u .\eea    
\end{defn}



The next lemma allows us the privilege of going back and forth between  the operator $T_{u\circ \bl \theta}|_{\mb P_\varrho(H^2(\Omega))}$ (cf. Equation \eqref{subToep}) and the Toeplitz operator $T_u$ on ${H^2_\varrho(\bl\theta(\Omega))}.$

\begin{lem}\label{equi1}
If  $u\in L^\infty(\del\bl \theta(\Omega)),$ then   the Toeplitz operator $T_u$ on ${H^2_\varrho(\bl \theta(\Omega))}$ is unitarily equivalent to the restriction of $T_{\widetilde{u}}$ to ${\mb P_\varrho(H^2(\Omega))}$ for every $\varrho \in \widehat{G}_1,$ where $\widetilde{u}= u\circ \bl \theta.$ 
\end{lem}
\begin{proof} The operator $T_{\widetilde u}: \mb P_\varrho(H^2(\Omega))\to \mb P_\varrho(H^2(\Omega))$ is given by the formula (cf. Equation \eqref{subToep})
$$T_{\widetilde u}(f) = \inner{\widetilde uf}{S_{\varrho}(\cdot,z)}_{L^2(\del\Omega)},$$ where $S_\varrho$ denotes the reproducing kernel of the subspace $\mb P_\varrho(H^2(\Omega))$ cf. Equation \eqref{repvar}.  
For $f\in \mb P_\varrho(H^2(\Omega))$ and  $z \in \Omega,$ it follows that (cf. Equation \eqref{gam} and \eqref{gamo}) 
$$0= \inner{uf-P_\varrho(u f)}{\Gamma_\varrho^{\ell*}S_\varrho(\cdot,z)}=\inner{\Gamma_\varrho^\ell(uf-P_\varrho(u f))}{S_\varrho(\cdot,z)}=\inner{\Gamma_\varrho^\ell(uf)-\Gamma_\varrho^h P_\varrho(u f)}{S_\varrho(\cdot,z)}.$$  If $f \in H^2_\varrho(\bl\theta(\Omega))$ then $\Gamma_\varrho^\ell(uf)=(u\circ \bl \theta)\Gamma_\varrho^h(f).$ Therefore, 
\Bea (\Gamma_\varrho^h T_u f)(z) = (\Gamma_\varrho^h P_\varrho(u f))(z) &=&\inner{\Gamma_\varrho^h P_\varrho(u f)}{S_\varrho(\cdot,z)} \\ &=& \inner{ \Gamma_\varrho^\ell(u f)}{S_\varrho(\cdot,z)} \\ 
&=& \inner{ (u\circ \bl \theta) \Gamma_\varrho^h(f)}{S_\varrho(\cdot,z)} = (T_{\widetilde{u}}\Gamma_\varrho^hf)(z), \Eea whence  the proof follows.
\end{proof} 

Let $\tr:G \to \mb C^*$ denote the trivial representation of the group $G.$  For any one-dimensional representation $\varrho \in \widehat{G}$ and $f \in \mb P_{\tr}(H^2(\Omega)),$ it follows that
$$ \ell_\varrho f \in \ell_\varrho \cdot \mb P_{\tr}(H^2(\Omega)) \subseteq R^G_\varrho(H^2(\Omega)) = \mb P_\varrho(H^2(\Omega)).$$ The density of $G$-invariant polynomials in $\mb P_{\tr}(H^2(\Omega))$ implies that $\ell_\varrho \cdot \mb P_{\tr}(H^2(\Omega))$ is dense in $\mb P_\varrho(H^2(\Omega)).$ 

The following lemma identifies a crucial invariant subspace for the operator $T_{\widetilde u}.$  
\begin{lem}\label{invlem} If $\widetilde{u} \in L^\infty(\del\Omega)$ is a $G$-invariant function, then the restriction of the operator $T_{\widetilde{u}}$ on ${\mb P_\varrho(H^2(\Omega))}$  leaves the subspace $\ell_\varrho \cdot \mb P_{\tr}(H^2(\Omega))$ invariant for every $\varrho\in\widehat G_1.$
\end{lem}
\begin{proof}
Let $f \in \mb P_\varrho(H^2(\Omega))$ be such that $f = \ell_\varrho f_\varrho$ for $f_\varrho \in \mb P_{\tr}(H^2(\Omega)).$  An appeal to Remark \ref{rem3} shows that  $\widetilde{u} f_\varrho \in \mb P_{\tr} (L^2(\del\Omega))$  and the following holds: 
\bea \nonumber (T_{\widetilde{u}}f)(z) = \inner{\widetilde{u} f}{S_\Omega(\cdot, z)} = \inner{\widetilde{u} f_\varrho}{M^*_{\ell_\varrho} S_\Omega(\cdot, z)} &=& \ell_\varrho(z) \inner{\widetilde{u} f_\varrho}{S_\Omega(\cdot, z)}  \\ 
&=&\nonumber \ell_\varrho(z) \inner{\widetilde{u} f_\varrho}{S_{\tr}(\cdot, z)}  \\
&=& \label{invar1} \ell_\varrho(z) \widetilde{P}_{\tr}(\widetilde{u} f_\varrho) (z), \eea where $S_{\tr}$ denotes the reproducing kernel of $\mb P_{\rm tr}(H^2(\Omega)).$ Therefore, we conclude that $$T_{\widetilde{u}}(\ell_\varrho \cdot \mb P_{\tr}(H^2(\Omega)) \subseteq \ell_\varrho \cdot \mb P_{\tr}(H^2(\Omega))
\text{~~for every~~} \varrho \in \widehat{G}_1.$$ 
\end{proof}
\begin{rem}
The conclusion of the preceding lemma can be extended to any representation $\varrho \in \widehat{G}$ with $\deg( \varrho) > 1.$ 
Since $\mb P_\varrho (\mb C[z_1,\ldots,z_n])$ is a free module over $\mb C[z_1,\ldots,z_n]^G$ of rank $(\deg \varrho)^2$ \cite{Pan}, there is a basis $\{\ell_{\varrho,i}:i=1,\ldots,(\deg \varrho)^2 \}$ of   $\mb P_\varrho (\mb C[z_1,\ldots,z_n])$ as a free module over $\mb C[z_1,\ldots,z_n]^G.$  By the density of  $\sum_{i=1}^{(\deg \varrho)^2} \ell_{\varrho,i} \cdot \mb C[z_1,\ldots,z_n]^G $ in $\sum_{i=1}^{(\deg\varrho)^2} \ell_{\varrho,i} \cdot \mb P_{\rm tr} (H^2(\Omega))$ and the fact that $\sum_{i=1}^{(\deg\varrho)^2} \ell_{\varrho,i} \cdot \mb P_{\rm tr} (H^2(\Omega))$ is contained in $\mb P_\varrho (H^2(\Omega)),$ it follows that $\sum_{i=1}^{(\deg \varrho)^2} \ell_{\varrho,i} \cdot \mb P_{\rm tr} (H^2(\Omega))$ is dense in $\mb P_\varrho (H^2(\Omega)).$ For $f = \sum_{i=1}^{(\deg \varrho)^2} \ell_{\varrho,i} f_{\varrho,i},$ such that  $f_{\varrho,i} \in \mb P_{\rm tr} (H^2(\Omega)),$  it follows that:
\bea\label{sum}
\nonumber(T_{\widetilde{u}}f)(z) = \inner{\widetilde{u} f}{S_\Omega(\cdot, z)} &=& \inner{\sum_{i=1}^{(\deg \varrho)^2} \ell_{\varrho,i} \widetilde{u} f_{\varrho,i}}{S_\Omega(\cdot, z)} \\&=&\nonumber \sum_{i=1}^{(\deg \varrho)^2} \inner{\widetilde{u} f_{\varrho,i}}{M^*_{\ell_{\varrho,i}}S_\Omega(\cdot, z)} \\&=& \nonumber \sum_{i=1}^{(\deg \varrho)^2} \ell_{\varrho,i}(z) \inner{\widetilde{u} f_{\varrho,i}}{S_\Omega(\cdot, z)} \\ &=&\nonumber \sum_{i=1}^{(\deg \varrho)^2} \ell_{\varrho,i}(z) \inner{\widetilde{u} f_{\varrho,i}}{S_{\tr}(\cdot, z)} \\ &=& \sum_{i=1}^{(\deg \varrho)^2} \ell_{\varrho,i}(z) \widetilde{P}_{\tr}(\widetilde{u} f_{\varrho,i})(z).
\eea
Hence  $\sum_{i=1}^{(\deg \varrho)^2} \ell_{\varrho,i} \cdot \mb P_{\rm tr} (H^2(\Omega))$ is left invariant by the operator $T_{\widetilde{u}}$ for every $\varrho\in\widehat G.$
\end{rem}
For a $G$-invariant function $\widetilde u\in L^\infty(\del\Omega),$ if $T_{\widetilde{u}}=0$ on $H^2(\Omega),$ then  $T_{\widetilde{u}}=0$ on $\mb P_\varrho(H^2(\Omega)).$ So $T_u =0 $ on $H^2_\varrho(\bl \theta(\Omega))$ for every $\varrho\in \widehat{G}_1$ by Lemma \ref{equi1}. It is interesting to note that the converse  also holds.
\begin{prop}\label{onetoone}
  Let $u\in L^\infty(\del\bl \theta(\Omega))$ and $\widetilde{u}=u \circ \bl \theta.$ If $T_u =0 $ on $H^2_\varrho(\bl \theta(\Omega))$ for some $\varrho \in \widehat{G}_1,$ then $T_{\widetilde{u}}=0$ on $H^2(\Omega).$ \end{prop}
\begin{proof}
Lemma \ref{equi1} shows that $T_{\widetilde{u}}=0$ on $\mb P_\varrho(H^2(\Omega)).$  It follows from Equation \eqref{invar1} that $$0=T_{\widetilde{u}}(\ell_\varrho f) = \ell_\varrho \widetilde{P}_{\tr}( {\widetilde{u}}f) {~~for~~ every~~} f \in\mb P_{\rm tr} (H^2(\Omega)).$$
Since $\ell_\varrho$ vanishes on a measure zero subset of $\Omega,$ $\widetilde{P}_{\tr}( {\widetilde{u}}f)=0$ for every $f \in\mb P_{\rm tr} (H^2(\Omega)).$ If $f=\sum_{\varrho \in \widehat{G}}\sum_{i=1}^{(\deg \varrho)^2} \ell_{\varrho,i} f_{\varrho,i}$ for $f_{\varrho,i} \in \mb P_{\rm tr} (H^2(\Omega)).$  Then 
$$T_{\widetilde{u}}(f)=\sum_{\varrho \in \widehat{G}}\sum_{i=1}^{(\deg \varrho)^2} \ell_{\varrho,i} \widetilde{P}_{\tr}(\widetilde{u} f_{\varrho,i})=0.$$ 
Since $\sum_{\varrho \in \widehat{G}}\sum_{i=1}^{(\deg \varrho)^2} \ell_{\varrho,i} \cdot \mb P_{\rm tr} (H^2(\Omega))$ is dense in $H^2(\Omega),$ the proof is complete.
\end{proof}

 As a consequence of Corollary \ref{irrcor}, $T_{\widetilde{u}}=0$ on $H^2(\Omega)$ implies that $\widetilde{u}=0$ whenever $\Omega$ is an irreducible bounded symmetric domain. An analogous result holds  for $T_{\widetilde{u}}$ on $H^2(\mb D^n)$ as well. This leads to the following interesting conclusion. 

\begin{cor}
There is a natural $*$-linear embedding of $L^\infty(\del\bl\theta(\Omega))$ into $\m B(H^2_\varrho(\bl \theta(\Omega)))$ given by $u\mapsto T_u,$ whenever $\Omega$ is an irreducible bounded symmetric domain or  the unit polydisc $\mb D^n.$ 
\end{cor}
\begin{proof}
It suffices to show the following: If for $u\in L^\infty(\bl \theta(\del \Omega)),$     $T_u =0 $ on $H^2_\varrho(\bl \theta(\del \Omega))$ for some $\varrho \in \widehat{G}_1,$ then $u=0$ almost everywhere.
    
    The hypothesis along with Proposition \ref{onetoone} yields $T_{\widetilde{u}}=0$ on $H^2(\Omega),$ consequently, $\widetilde{u}=u\circ \bl \theta=0$.
\end{proof}  

\subsection{Multiplicative properties}\begin{lem}\label{zerlem}
    If  $\varrho \in \widehat{G}$ and $\widetilde u ,\widetilde v, \widetilde q\in L^\infty(\del\Omega)$ are $G$-invariant, then the following statements hold:
    \begin{enumerate}
        \item If $T_{\widetilde{u}}T_{\widetilde{v}}=T_{\widetilde{q}}$ on $\mb P_\varrho(H^2(\Omega)),$ then $T_{\widetilde{u}}T_{\widetilde{v}}=T_{\widetilde{q}}$ on $H^2(\Omega).$
        \item If $T_{\widetilde{u}}T_{\widetilde{v}}=T_{\widetilde{v}}T_{\widetilde{u}}$ on $\mb P_\varrho(H^2(\Omega)),$ then $T_{\widetilde{u}}T_{\widetilde{v}}=T_{\widetilde{v}}T_{\widetilde{u}}$ on $H^2(\Omega).$
    \end{enumerate}\end{lem}
\begin{proof}
    Consider an element $f = \sum_{i=1}^{(\deg\varrho)^2} \ell_{\varrho,i} \widehat{f}_{\varrho,i}$ in $\mb P_\varrho(H^2(\Omega)),$ where $\widehat{f}_{\varrho,i} \in \mb P_{\tr}(H^2(\Omega)).$ Then $T_{\widetilde{u}}T_{\widetilde{v}}f=T_{\widetilde{q}}f$ along with Equation \eqref{sum} implies that   \Bea \sum_{i=1}^{(\deg \varrho)^2} \ell_{\varrho,i} \widetilde{P}_{\tr}(\widetilde{u} ~ \widetilde{P}_{\tr}(\widetilde{v}\widehat{f}_{\varrho,i})) =\sum_{i=1}^{(\deg \varrho)^2}\ell_{\varrho,i} \widetilde{P}_{\tr}(\widetilde{q} \widehat{f}_{\varrho,i}).\Eea In case, $\varrho$ is one-dimensional, $\ell_{\varrho,1}=\ell_{\varrho}$ and thus for some $\widehat{f}_{\varrho} \in \mb P_{\tr}(H^2(\Omega)),$ $\ell_{\varrho} \widetilde{P}_{\tr}(\widetilde{u} ~ \widetilde{P}_{\tr}(\widetilde{v}\widehat{f}_{\varrho}) =\ell_{\varrho} \widetilde{P}_{\tr}(\widetilde{q} \widehat{f}_{\varrho}).$  This is equivalent to showing that $$\widetilde{P}_{\tr}(\widetilde{u}~ \widetilde{P}_{\tr}(\widetilde{v} \widehat{f}))= \widetilde{P}_{\tr}(\widetilde{q} \widehat{f}) \,\,\text{~for every~} \widehat{f} \in \mb P_{\tr}(H^2(\Omega)).$$
    If $\deg\varrho>1,$ we take $\widehat{f}_{\varrho,i} = 0$ for $i=2,\ldots,(\deg \varrho)^2$ and then repeat the argument analogous as above to arrive at the same conclusion. Therefore, the equality  $T_{\widetilde{u}}T_{\widetilde{v}}=T_{\widetilde{q}}$ holds on the dense subset $\sum_{\varrho \in \widehat{G}}\sum_{i=1}^{(\deg \varrho)^2} \ell_{\varrho,i} \cdot \mb P_{\rm tr} (H^2(\Omega))$ of $H^2(\Omega)$ which proves the first claim.

    Using Equation \eqref{sum} and  argument analogous as above one  concludes that $$\widetilde{P}_{\tr}(\widetilde{u}~ \widetilde{P}_{\tr}(\widetilde{v} \widehat{f}))=\widetilde{P}_{\tr}(\widetilde{v}~ \widetilde{P}_{\tr}(\widetilde{u} \widehat{f})) \,\,\text{~for every~} \widehat{f} \in \mb P_{\tr}(H^2(\Omega)).$$ 
    Let $f = \sum_{\varrho \in \widehat{G}} \sum_{i=1}^{(\deg \varrho)^2} \ell_{\varrho,i} f_{\varrho,i},$ for $f_{\varrho,i} \in \mb P_{\tr}(H^2(\Omega)),$ then \Bea T_{\widetilde{u}}T_{\widetilde{v}}f &=&\sum_{\varrho \in \widehat{G}} \sum_{i=1}^{(\deg \varrho)^2} \ell_{\varrho,i} \widetilde{P}_{\tr}(\widetilde{u}~ \widetilde{P}_{\tr}(\widetilde{v} f_{\varrho,i}))\\&=&\sum_{\varrho \in \widehat{G}} \sum_{i=1}^{(\deg \varrho)^2} \ell_{\varrho,i} \widetilde{P}_{\tr}(\widetilde{v}~ \widetilde{P}_{\tr}(\widetilde{u} f_{\varrho,i})) = T_{\widetilde{v}}T_{\widetilde{u}}f\Eea on a dense subset of $H^2(\Omega).$ This completes the proof. 
\end{proof}
\begin{rem}\label{rem1}\rm We isolate some of the key ingredients to prove the main results.
   \begin{enumerate}
\item 
An immediate consequence of part {\it 1.} of Lemma \ref{zerlem}  is that if $T_{\widetilde{u}}T_{\widetilde{v}}=T_{\widetilde{q}}$ on $\mb P_\varrho(H^2(\Omega))$ for at least one $\varrho \in \widehat{G}$ (irrespective of the degree of $\varrho$), then $T_{\widetilde{u}}T_{\widetilde{v}}=T_{\widetilde{q}}$ on $\mb P_\mu(H^2(\Omega))$ for every $\mu \in \widehat{G}.$
    \item Similarly, if $T_{\widetilde{u}}T_{\widetilde{v}}=T_{\widetilde{v}}T_{\widetilde{u}}$ on $\mb P_\varrho(H^2(\Omega))$ for at least one $\varrho \in \widehat{G},$ then $T_{\widetilde{u}}T_{\widetilde{v}}=T_{\widetilde{v}}T_{\widetilde{u}}$ on $\mb P_\mu(H^2(\Omega))$ for every $\mu \in \widehat{G}.$
\end{enumerate} 
\end{rem}
Now we are set to prove one of the main results of this paper.

\begin{proof}[Proof of Theorem \ref{zerthm}]
Assume  that  $T_uT_v=T_q$ on $H^2_{\mu}(\bl \theta(\Omega))$  for a one-dimensional representation $\mu$ of $G.$ Then using Lemma \ref{equi1}, one gets $T_{\widetilde{u}}T_{\widetilde{v}}=T_{\widetilde{q}}$ on $\mb P_\mu(H^2(\Omega)).$ By Remark \ref{rem1}, we have $T_{\widetilde{u}}T_{\widetilde{v}}=T_{\widetilde{q}}$ on $\mb P_\varrho(H^2(\Omega))$ for every $\varrho \in \widehat{G}_1.$ Lemma \ref{equi1} yields $T_uT_v=T_q$ on $H^2_\varrho(\bl \theta(\Omega))$ for every $\varrho\in \widehat{G}_1.$ Lastly, Lemma \ref{zerlem} concludes the rest.

Conversely, if $T_{\widetilde{u}}T_{\widetilde{v}}=T_{\widetilde{q}}$ on $H^2(\Omega),$ then $T_{\widetilde{u}}T_{\widetilde{v}}=T_{\widetilde{q}}$ on $\mb P_\mu(H^2(\Omega))$ for every  $\mu \in \widehat{G}.$ Then using Lemma \ref{equi1}, we infer the result.
\end{proof}
The proof of Theorem \ref{zerthm1} is very similar as above, thus omitted.

\begin{thm}{\rm (Finite zero-product property)}\label{finzer}
    Let $\varrho\in \widehat{G}_1$ and $u_i\in L^\infty(\del\bl \theta(\Omega))$ for $i=1,\ldots,k.$ The finite product of Toeplitz operators $T_{u_1}\ldots T_{u_k}=0$ on $H^2_\varrho(\bl\theta(\Omega))$ if and only if $T_{\widetilde u_1}\ldots T_{\widetilde u_k}=0$ on $H^2(\Omega),$ where $\widetilde u_i=u_i\circ \bl \theta$ for $i=1,\ldots,k.$
\end{thm}
\begin{proof}
Following the similar line of proof as above, we conclude from the hypothesis that  $$\widetilde{P}_{\tr}(\widetilde u_1\widetilde{P}_{\tr}(\widetilde u_2 \cdots \widetilde{P}_{\tr}(\widetilde u_k\widehat{f})\cdots)=0$$ for every $\widehat{f} \in \mb P_{\tr}(H^2(\Omega)),$ where $\widetilde u_i=u_i\circ \bl \theta$ for $i=1,\ldots,k.$ The result follows from density of $\sum_{\varrho \in \widehat{G}}\sum_{i=1}^{(\deg \varrho)^2} \ell_{\varrho,i} \cdot \mb P_{\rm tr} (H^2(\Omega))$ in $H^2(\Omega).$
\end{proof}
\subsection{On proper images of the unit ball and the polydisc}\label{specifically}Theorem \ref{zerthm} enables us to apply characterization of Toeplitz operators on $H^2(\Omega)$ (for example, commuting or semi-commuting pairs etc.) to specify conditions for characterizing Toeplitz operators on $H^2(\bl \theta(\Omega))$ (with the same property). The following results are an interesting depiction of it.

 We start by recalling that a function $\phi$ is called pluriharmonic in $\Omega$ if $$\frac{\del^2\phi}{\del z_i \del\overbar{z}_j}=0~\text{for  all}~i,j=1,\ldots,n.$$

\begin{defn}\cite[Definition 5.3]{GN23}
Let $\Omega\subseteq \mb C^n$ be a $G$-invariant domain and $\bl \theta: \Omega \to \bl \theta(\Omega)$ be a basic polynomial map associated to the finite complex reflection group $G.$ A function $\phi$ defined on $\bl \theta(\Omega)$ is said to be $G$-pluriharmonic on $\bl \theta(\Omega)$ if 
$\phi \circ \bl \theta$ is a pluriharmonic function on $\Omega.$
\end{defn}
Suppose that $\widetilde{\phi}$ is a pluriharmonic function on $\Omega.$ Then we write $\phi \circ \bl \theta  = \sum_{\sigma \in G} \widetilde{\phi} \circ \sigma$ and $\phi$ is a $G$-pluriharmonic function on $\bl \theta(\Omega).$  

\subsubsection{For the unit ball}

Recall that two Toeplitz operators on $H^2(\mb D)$ commute if and only if either both are analytic, or both are co-analytic, or one is a linear function of the other \cite[p. 98, Theorem 9]{BH1964}. An analogous result for Toeplitz operators with bounded pluriharmonic symbols on $H^2(\mb B_n)$ can be found in \cite[Theorem 2.2]{Zheng98}. We combine \cite[Theorem 2.2]{Zheng98} and Theorem \ref{zerthm1} to conclude the following:
\begin{thm}\label{bncom}
    Let $u$ and $v$ be two bounded $\mb Z_m$-pluriharmonic functions on $\mathcal E_n(m)$. Then $T_u T_v = T_v T_u$
on the Hardy space $H^2(\m E_n(m))$ if and only if $u$ and $v$ satisfy one of the following conditions:
\begin{enumerate}
    \item[1.]  Both $u$ and $v$ are holomorphic on $\mathcal E_n(m)$.
    \item[2.] Both $\ov{u}$ and $\ov{v}$ are holomorphic on $\mathcal E_n(m)$.
    \item[3.] Either $u$ or $v$ is constant on $\mathcal E_n(m)$.
    \item[4.] There is a nonzero constant $b$ such that $u - b v$ is constant on $\mathcal E_n(m)$.
\end{enumerate}
\end{thm}

\subsubsection{For the polydisc}
We refer to the proper images of the open unit polydisc by $\bl\theta(\mb D^n).$ It is understood that  $G$ is a finite complex reflection group acting on $\mb D^n$ and $\bl\theta:\mb D^n\to\bl\theta(\mb D^n)$ is a basic polynomial map associated to the group $G.$  

The description for commuting pairs and semi-commuting pairs of Toeplitz operators on $H^2(\mb D^2)$ can be found on \cite[p. 3336, Theorem 1.4]{DSZ12} and \cite[p. 176, Theorem 2.1]{GZ97}, respectively. We combine it with Theorem \ref{zerthm} to classify commuting pairs and semi-commuting pairs of Toeplitz operators on $H^2(\bl \theta(\mb D^2)).$

\begin{notation}
    For $f,g \in L^\infty(\mb T^2),$ we define $D_i(f,g):=\frac{\del f}{\del z_i}  \frac{\del g}{\del \ov{z_i}},$ $i=1,2.$ Also, $D_{1,2}(f,g):=\frac{\del^2 f}{\del z_1\del z_2}\frac{\del^2 g}{\del \ov{z_1}\del \ov{z_2}}.$ 
\end{notation}
\begin{thm}
   Let $u,v \in L^\infty(\bl\theta(\mb T^2)).$ Then $T_uT_v =T_vT_u$ on $H^2(\bl\theta(\mb D^2))$ if and only if the following conditions hold:
   \begin{enumerate}
       \item[1.] For almost all $\xi \in \mb T,$ $D_1(u\circ\bl\theta,v\circ\bl\theta)(z,\xi)=D_1(v\circ\bl\theta,u\circ\bl\theta)(z,\xi)$ for all $z\in \mb D.$
       \item[2.] For almost all $\xi \in \mb T,$ $D_2(u\circ\bl\theta,v\circ\bl\theta)(\xi,z)=D_2(v\circ\bl\theta,u\circ\bl\theta)(\xi,z)$ for all $z\in \mb D.$
       \item[3.] For every $z_1,z_2\in \mb D^2,$ $D_{1,2}(u\circ\bl\theta,v\circ\bl\theta)(z_1,z_2)=D_{1,2}(v\circ\bl\theta,u\circ\bl\theta)(z_1,z_2).$
   \end{enumerate}
\end{thm}

\begin{thm}\label{semd2}
   Let $u,v \in L^\infty(\bl\theta(\mb T^2)).$ Then $T_uT_v =T_{uv}$ on $H^2(\bl\theta(\mb D^2))$ if and only if the following conditions hold:
   \begin{enumerate}
       \item[1.] For almost all $\xi \in \mb T,$ $D_1(u\circ\bl\theta,v\circ\bl\theta)(z,\xi)=0$ for all $z\in \mb D.$
       \item[2.] For almost all $\xi \in \mb T,$ $D_2(u\circ\bl\theta,v\circ\bl\theta)(\xi,z)=0$ for all $z\in \mb D.$
       \item[3.] For every $z_1,z_2\in \mb D,$ $D_{1,2}(u\circ\bl\theta,v\circ\bl\theta)(z_1,z_2)=0.$
   \end{enumerate}
\end{thm}
Equivalently, we have the following from Theorem \ref{zerthm} and \cite[p. 176, Theorem 2.1]{GZ97}.
\begin{prop}
    Let $u,v \in L^\infty(\bl\theta(\mb T^2)).$ Then $T_uT_v =T_{uv}$ on $H^2(\bl\theta(\mb D^2))$ if and only if for each $i=1,2;$ either $\ov{u\circ \bl\theta}$ or $v\circ \bl\theta$ is holomorphic in $z_i.$
\end{prop}

\cite[p. 190, Main Theorem]{Lee10} provides a characterization of commuting pairs of Toeplitz operators on $H^2(\mb D^n).$ One can apply Theorem \ref{zerthm1} in combination with \cite[p. 190, Main Theorem]{Lee10} to describe all commuting pairs of Toeplitz operators on $H^2(\bl\theta(\mb D^n)).$ We close our discussion on multiplicative properties of Toeplitz operators on specific domains here. There is a vast literature  in this direction for various bounded symmetric domains and using those results, the similar observations are possible for Toeplitz operators on their proper images as well.
 
\subsection{Brown-Halmos type characterization}
We now specialize $\Omega$ to be the open unit polydisc $\mb D^n$ and prove a Brown-Halmos type characterization of Toeplitz operators on $H^2(\bl\theta(\mb D^n)),\,\,\bl\theta$ being a basic polynomial mapping associated to $G(m.p,n),$  where $ m,n,p$ are  positive integers, $n>1$ and $p$ divides $m.$ Let $q=m/p.$ Recall form  Example \ref{exo} that \bea\label{sym}\nonumber \theta_i(z)&=&s_i(z_1^m,\ldots,z_n^m) \text{~for~} i=1,\ldots,n-1 \,\,\\ \text{~and~}\,\, \theta_n(z) &=& (z_1\cdots z_n)^{q}.\eea 
 Moreover,  \bea\label{eqn} \ov{\theta_i}(z) \theta_n^p(z) = \theta_{n-i}(z)\,\, \text{~for~}  z \in \mb T^n \text{~and~} i=1,\ldots,n-1 .\eea

  Henceforth, for the sake of simplicity, we write $P:L^2_{\rm sgn}(\bl \theta(\mb T^n))\to H^2(\bl\theta(\mb D^n))$ for the orthogonal projection. For $u\in L^\infty(\bl \theta(\mb T^n)),$ the Toeplitz operator $T_u: H^2(\bl\theta(\mb D^n)) \to H^2(\bl\theta(\mb D^n))$ is given by 
  $$T_u = PM_u.$$ For $\widetilde{u}=u\circ\bl \theta,$ the subspace $\mb P_{\rm sgn}(L^2(\mb T^n))$ is reducing for the Laurent operator $M_{\widetilde{u}} : L^2(\mb T^n) \to L^2(\mb T^n).$ We recall that the unitary $\Gamma_{\rm sgn}^\ell : L^2_{\rm sgn}(\bl \theta(\mb T^n)) \to \mb P_{\rm sgn}(L^2(\mb T^n))$ intertwines the Laurent operators $M_u$ and $M_{\widetilde{u}},$ that is, $\Gamma_{\rm sgn}^\ell M_u=M_{\widetilde{u}}\Gamma_{\rm sgn}^\ell.$ In particular, for the $i$-th coordinate multiplication $M_i$ on $L^2_{\rm sgn}(\bl \theta(\mb T^n)),$ $$\Gamma_{\rm sgn}^\ell M_i = M_{\theta_i} \Gamma_{\rm sgn}^\ell,$$ where $M_{\theta_i}$ denotes the multiplication operator by the polynomial $\theta_i$ on $\mb P_{\rm sgn}(L^2(\mb T^n)).$ 

Similarly, the unitary operator $\Gamma_{\rm sgn}^h : H^2(\bl \theta(\mb D^n)) \to \mb P_{\rm sgn}(H^2(\mb D^n))$ intertwines $T_u$ and $T_{\widetilde{u}}.$ In particular,  
If  $T_i : H^2(\bl \theta(\mb D^n)) \to H^2(\bl \theta(\mb D^n))$  is the $i$-th coordinate multiplication defined by $$(T_i f)(z) = z_i f(z) \text{ for } z \in \bl \theta(\mb D^n),$$ 
then $\Gamma_{\rm sgn}^h T_i = T_{\theta_i} \Gamma_{\rm sgn}^h.$
 \begin{rem}
     We mostly use the notation $\theta_i$ for denoting both the $i$-th component of the proper holomorphic map $\bl \theta : \Omega \to \bl \theta(\Omega)$ and the associated polynomial in $n$ variables. Although while writing $M_{\theta_i},$ it should be understood as the multiplication operator by the polynomial $\theta_i,$ without any ambiguity.
 \end{rem}
To prove the following results, we largely follow the constructions  in \cite{BH1964} which was suitably adapted for several variables in \cite{DJ77} and \cite{BDS2021}. We start with a couple of lemmas. 
\begin{lem}\label{prelem}
  For  $i=1,\ldots,n-1,$  $M_i^* M_n^p = M_{n-i}$ on $L^2_{\rm sgn}(\bl \theta(\mb T^n))$ and $T_i^* T_n^p = T_{n-i}$ on $H^2(\bl\theta(\mb D^n)).$
\end{lem}
\begin{proof}
    The first part follows from Equation \eqref{eqn} and the unitary equivalence of $M_i|_{L^2_{\rm sgn}(\bl \theta(\mb T^n))}$ and $M_{\theta_i}|_{\mb P_{\rm sgn}(L^2(\mb T^n))}$. For the second part, we  observe for  $f,g \in H^2(\bl\theta(\mb D^n))$  that $$\inner{T_i^* T_n^p f}{g}=\inner{ T_n^p f}{T_ig}= {\inner{ M_n^p f}{M_ig}}_{L^2}= {\inner{M_i^*M_n^p f}{g}}_{L^2}={\inner{M_{n-i} f}{g}}_{L^2}= \inner{T_{n-i} f}{g}. $$
\end{proof}  
\begin{lem}\label{lem1}
    Let $T$ be a bounded operator on $L^2_{\rm sgn}(\bl \theta(\mb T^n))$ which commutes with  $M_i$ for $i=1,\ldots,n.$ Then there exists a function $\phi \in L^\infty(\bl \theta(\mb T^n))$ such that $T=M_\phi.$
\end{lem}
\begin{proof}
We first note that $(M_1,\dots,M_n)$ is an $n$-tuple of commuting normal operators on $L^2_{\rm sgn}(\bl \theta(\mb T^n))$. Therefore, the Taylor joint spectrum of $(M_1,\dots,M_n)$ is $\bl \theta(\mb T^n).$ Invoking the spectral theorem for commuting normal operators, it follows that the von Neumann algebra generated by $(M_1,\dots,M_n)$ is given by the algebra $L^\infty(\bl \theta(\mb T^n)).$ Since it is a maximal von Neumann algebra, the commutant algebra of $(M_1,\dots,M_n)$ is $*$-isomorphic to $L^\infty(\bl \theta(\mb T^n)).$ Hence  the result follows.
\end{proof}
 Let $T$ be a bounded operator on $\mb P_{\rm sgn}(L^2(\mb T^n))$ which commutes with $M_{\theta_i}$ for $i=1,\ldots,n.$ An appeal to Lemma \ref{equi1} and Lemma \ref{lem1} allows us to conclude that there exists a $G$-invariant function $\widetilde{\phi} \in L^\infty(\mb T^n)$ such that $T=M_{\widetilde{\phi}}.$

 Now we are ready to prove a Brown-Halmos type characterization of Toeplitz operators on $H^2(\bl\theta(\mb D^n)),$  $\bl\theta$ being a basic polynomial associated to the group $G(m,p,n)$ satisfying Equation \eqref{eqn}.

\begin{proof}[Proof of Theorem \ref{bhthm}]
Let $T=T_\phi$ with $\phi\in L^\infty(\bl\theta(\mb T^n)).$ Then 
$$\inner{T^*_i T_\phi T^p_nf}{g}= \inner{ T_\phi T^p_nf}{T_ig}= \inner{ M_\phi M_n^pf}{M_ig}_{L^2}=\inner{  M^*_{i}M_n^pM_\phi f}{g}_{L^2}.$$
Consequently,  it follows from Lemma \ref{prelem} that
$$\inner{T^*_i T_\phi T^p_nf}{g}= \inner{  M_{n-i}M_\phi f}{g}_{L^2}=\inner{PM_\phi M_{n-i} f}{g}_{H^2}=\inner{T_\phi T_{n-i} f}{g}.$$
Since  $$\inner{T^*_n T_\phi T_nf}{g}=\inner{ T_\phi T_nf}{T_ng}=\inner{ M_\phi M_nf}{M_ng}_{L^2}=\inner{ M_\phi f}{g}_{L^2}=\inner{ T_\phi f}{g},$$
the  second condition follows. 

For the converse, we work on $\mb P_{\rm sgn}(L^2(\mb T^n)).$ Depending on the group $G=G(m,p,n),$ there exists a subset $\m I_{G}$ of $\mb Z^n$ such that 
$$\{\gamma_{\bl m}(z) := \sqrt{|G|}\mb P_{\rm sgn} z^{\bl m}: \bl m \in \m I_G\}$$
forms an orthogonal basis of $\mb P_{\rm sgn}(L^2(\mb T^n)),$ cf. Example \ref{onb}. Let 
\bea\label{holset} \m I_{G,{\rm hol}}:= \m I_G \cap \mb N_{0}^n.\eea
The set $\{\gamma_{\bl m} : \bl m \in  \m I_{G,{\rm hol}} \}$ forms an orthogonal basis for $\mb P_{\rm sgn}(H^2(\mb D^n)).$  

Recall that $q=m/p.$ Since for every $r\geq 0,$ $\theta_n^r(z_1,\ldots,z_n)=(z_1\cdots z_n)^{qr}$ is a $G$-invariant polynomial, 
$$\theta_n^r(z) \gamma_{\bl m}(z)= \sqrt{|G|}\theta_n^r(z) \mb P_{\rm sgn} z^{\bl m} = \sqrt{|G|} \mb P_{\rm sgn} \theta_n^r(z)z^{\bl m} = \sqrt{|G|}\mb P_{\rm sgn} z^{\bl m +q\bl r} = \gamma_{\bl m + q\bl r}(z),$$ where $q \bl r$ denotes the $n$-tuple $(qr,\ldots,qr).$ 

By hypothesis and Lemma \ref{equi1}, there exists a bounded operator $\widetilde{T}$ on $\mb P_{\rm sgn}(H^2(\mb D^n))$ which is unitarily equivalent to $T$ on $H^2(\bl \theta(\mb D^n))$ and satisfies the conditions 
$$T_{\theta_i}^*\widetilde{T}T_{\theta_n}^p = \widetilde{T} T_{\theta_{n-i}} \text{~~~~and~~~~} T_{\theta_n}^*\widetilde{T}T_{\theta_n}=\widetilde{T}.$$ 
Next, we show that $\widetilde{T}=T_{\widetilde {\phi}}$ for a $G$-invariant symbol $\widetilde {\phi}$ in $L^\infty( \mb T^n).$ Clearly, for every non-negative integer $r,$  $T_{\theta_n}^{r*}\widetilde{T}T_{\theta_n}^{r}=\widetilde{T}.$ Hence $$\inner{\widetilde{T}\gamma_{\bl p}}{\gamma_{\bl m}}= \inner{\widetilde{T}T_{\theta_n}^{r}\gamma_{\bl p}}{T_{\theta_n}^{r}\gamma_{\bl m}}=\inner{\widetilde{T}\gamma_{\bl p+q\bl r}}{\gamma_{\bl m+q\bl r}} \text{~~~~for every~~~~} r\geq 0 \text{~~~~and~~~~} \bl p,\bl m \in \m I_{G,{\rm hol}}.$$

For every non-negative integer $r,$ we define an operator $A_r=M^{*r}_{\theta_n}\widetilde{T}\widetilde{P}_{\rm sgn}M^r_{\theta_n}$ on $\mb P_{\rm sgn}(L^2(\mb T^n)),$ $\widetilde{P}_{\rm sgn}: \mb P_{\rm sgn}(L^2(\mb T^n)) \to \mb P_{\rm sgn}(H^2(\mb D^n))$ being the associated orthogonal projection. Note that 
$$M_{\theta_n}^r \gamma_{\bl p} = \gamma_{\bl p+q\bl r} \in \mb P_{\rm sgn}(L^2(\mb T^n)) \text{~~for~~} \bl p\in \m I_{G} {~~and~~} r>0.$$ 
Moreover, for every $\bl p \in \m I_{G},$ there exists a sufficiently large $r$ (depending on $\bl p$) such that 
$$M_{\theta_n}^r \gamma_{\bl p} = \gamma_{\bl p+q\bl r} \in \mb P_{\sgn}(H^2(\mb D^n)).$$ 
For sufficiently large $r,$ it follows that 
$$\inner{A_r \gamma_{\bl p}}{\gamma_{\bl m}}= \inner{\widetilde{T}\widetilde{P}_{\rm sgn}M^r_{\theta_n} \gamma_{\bl p}}{M^{r}_{\theta_n}\gamma_{\bl m}}=\inner{\widetilde{T}\gamma_{\bl p+q\bl r}}{\gamma_{\bl m+q\bl r}} \text{~~for~~} \bl p,\bl m \in \m I_{G}.$$  Therefore, if $\phi_1$ and $\phi_2$ are finite linear combinations of $\gamma_{\bl m}$'s for $\bl m \in \m I_G,$ then $\{\inner{A_r \phi_1}{\phi_2}\}$ is convergent. Also, for every  $r\geq 0,$ we have $\norm{A_r} \leq \norm{A_0}=\Vert\widetilde{T}\Vert$ which implies that $\{A_r\}$ converges in  weak  operator topology to a bounded operator, say $A_\infty,$ on $\mb P_{\rm sgn}(L^2(\mb T^n)).$ To prove that $A_\infty$ commutes with each $M_{\theta_i},$ we first observe that $A_\infty$ commutes with $M_{\theta_n}$ and for $1\leq i \leq n-1,$ it follows that
\Bea \inner{M_{\theta_i}^* A_\infty^* \gamma_{\bl p}}{\gamma_{\bl m}} &=& \lim_{r} \inner{M_{\theta_i}^* A_r^* \gamma_{\bl p}}{\gamma_{\bl m}}\\ 
&=& \lim_{r} \inner{M_{\theta_i}^* M_{\theta_n}^{*r} \widetilde{T}^* \widetilde{P}_{\rm sgn} M_{\theta_n}^{r}\gamma_{\bl p}}{\gamma_{\bl m}} \\ &=&
\lim_{r} \inner{M_{\theta_i}^*  \widetilde{T}^* \widetilde{P}_{\rm sgn}M_{\theta_n}^{r}\gamma_{\bl p}}{M_{\theta_n}^{r}\gamma_{\bl m}} \\ &=&
\lim_{r} \inner{T_{\theta_i}^{*}  \widetilde{T}^* \widetilde{P}_{\rm sgn}M_{\theta_n}^{r}\gamma_{\bl p}}{M_{\theta_n}^{r}\gamma_{\bl m}} \\ &=& \lim_{r} \inner{T_{\theta_n}^{*p}  \widetilde{T}^* T_{\theta_{n-i}} \widetilde{P}_{\rm sgn}M_{\theta_n}^{r}\gamma_{\bl p}}{M_{\theta_n}^{r}\gamma_{\bl m}} \\ &=& \lim_{r} \inner{M_{\theta_n}^{*p}  \widetilde{T}^* \widetilde{P}_{\rm sgn}M_{\theta_n}^{r} M_{\theta_{n-i}} \gamma_{\bl p}}{M_{\theta_n}^{r}\gamma_{\bl m}}\\ &=& \lim_{r} \inner{M_{\theta_n}^{*(r+p)}  \widetilde{T}^* \widetilde{P}_{\rm sgn}M_{\theta_n}^{(r+p)}M_{\theta_n}^{*p} M_{\theta_{n-i}} \gamma_{\bl p}}{\gamma_{\bl m}}\\ &=& \lim_{r} \inner{M_{\theta_n}^{*(r+p)}  \widetilde{T}^* \widetilde{P}_{\rm sgn}M_{\theta_n}^{(r+p)}M^*_{\theta_{i}} \gamma_{\bl p}}{\gamma_{\bl m}} \\ &=& \inner{ A_\infty^*M_{\theta_i}^* \gamma_{\bl p}}{\gamma_{\bl m}}.\Eea 
Thus, $A_\infty$ commutes with each $M_{\theta_i}.$ Hence  there exists a $G$-invariant function $\widetilde{\phi}$ in $L^\infty(\mb T^n)$ such that $A_\infty = M_{\widetilde{\phi}}.$ Let $f,g \in \mb P_{\rm sgn}(H^2(\mb D^n)),$ then \Bea \inner{\widetilde{P}_{\rm sgn}M_{\widetilde{\phi}} f}{g}  = \inner{A_\infty f}{g} &=& \lim_{r}\inner{A_r f}{g}\\ &=& \lim_{r} \inner{M_{\theta_n}^{*r} \widetilde{T} \widetilde{P}_{\rm sgn} M_{\theta_n}^{r}f}{g} \\ &=& \lim_{r} \inner{ \widetilde{T} T_{\theta_n}^{r}f}{T_{\theta_n}^{r}g} \\ &=& \inner{ \widetilde{T}f}{g}. \Eea 
Thus, $\widetilde T=\widetilde{P}_{\rm sgn}M_{\widetilde{\phi}}=T_{\widetilde \phi},$ where $\widetilde{\phi} =\phi \circ \bl \theta$ for $\phi \in L^\infty(\bl \theta(\mb T^n)).$
This completes the proof.
\end{proof} 
We conclude this paper by  characterizing  the compact Toeplitz operators on $H^2(\bl\theta(\mb D^n)).$
\begin{thm} The only compact Toeplitz operator on $H^2(\bl\theta(\mb D^n))$ is the zero operator.
\end{thm}

 \begin{proof}
   Recall the definition of $\m I_{G,{\rm hol}}$ from Equation \eqref{holset}.  For some $\bl m,\bl p \in \m I_{G,{\rm hol}},$ it follows that
     \Bea \inner{T_{u\circ \bl \theta} \gamma_{\bl p}}{\gamma_{\bl m}} = \inner{T^{r*}_nT_{u\circ \bl \theta} T^r_n \gamma_{\bl p}}{\gamma_{\bl m}} =\inner{T_{u\circ \bl \theta} \gamma_{\bl p+q\bl r}}{\gamma_{\bl m+q\bl r}} \text{ (for every } r\geq 0).\Eea
     Since $T_u$ is compact, $\norm{T_{u\circ \bl \theta}\gamma_{\bl p}}$ goes to $0$ as $\bl p$ goes to infinity. Hence from the above, we have \Bea |\inner{T_{u\circ \bl \theta} \gamma_{\bl p}}{\gamma_{\bl m}}| = |\inner{T_{u\circ \bl \theta} \gamma_{\bl p+q\bl r}}{\gamma_{\bl m+q\bl r}}| \leq  \norm{T_{u\circ \bl \theta}\gamma_{\bl p +q\bl r}} \to 0\Eea as $r \to 0.$ Since $\bl m,\bl p \in \m I_{G,{\rm hol}}$ are chosen arbitrarily, $u$ is identically zero.
 \end{proof}

 \subsection*{Conflict of interest} The authors declare that there is no conflict of interest.



\bibliographystyle{siam}
\bibliography{Bibliography.bib}

\begin{thebibliography}{10}

\bibitem{AWY07}
{\sc A.~A. Abouhajar, M.~C. White, and N.~J. Young}, {\em A {S}chwarz lemma for
  a domain related to {$\mu$}-synthesis}, J. Geom. Anal., 17 (2007),
  pp.~717--750.

\bibitem{AMc}
{\sc J.~Agler and J.~E. McCarthy}, {\em Pick interpolation and {H}ilbert
  function spaces}, vol.~44 of Graduate Studies in Mathematics, American
  Mathematical Society, Providence, RI, 2002.

\bibitem{Arazy}
{\sc J.~Arazy}, {\em A survey of invariant {H}ilbert spaces of analytic
  functions on bounded symmetric domains}, in Multivariable operator theory
  ({S}eattle, {WA}, 1993), vol.~185 of Contemp. Math., Amer. Math. Soc.,
  Providence, RI, 1995, pp.~7--65.

\bibitem{BB85}
{\sc E.~Bedford and S.~Bell}, {\em Boundary behavior of proper holomorphic
  correspondences}, Math. Ann., 272 (1985), p.~505–518.

\bibitem{BD85}
{\sc E.~Bedford and J.~Dadok}, {\em Proper holomorphic mappings and real
  reflection groups}, J. Reine Angew. Math., 361 (1985), p.~162–173.

\bibitem{BDS2021}
{\sc T.~Bhattacharyya, B.~K. Das, and H.~Sau}, {\em Toeplitz operators on the
  symmetrized bidisc}, Int. Math. Res. Not. IMRN,  (2021), pp.~8492--8520.

\bibitem{BDGS22}
{\sc S.~Biswas, S.~Datta, G.~Ghosh, and S.~Shyam~Roy}, {\em Reducing submodules
  of {H}ilbert modules and {C}hevalley-{S}hephard-{T}odd theorem}, Adv. Math.,
  403 (2022).

\bibitem{BH1964}
{\sc A.~Brown and P.~R. Halmos}, {\em Algebraic properties of {T}oeplitz
  operators}, J. Reine Angew. Math., 213 (1963/64), pp.~89--102.

\bibitem{E35}
{\sc E.~Cartan}, {\em Sur les domaines born\'{e}s homog\`enes de l'espace den
  variables complexes}, Abh. Math. Sem. Univ. Hamburg, 11 (1935), pp.~116--162.

\bibitem{Cos05}
{\sc C.~Costara}, {\em On the spectral {N}evanlinna-{P}ick problem}, Studia
  Math., 170 (2005), pp.~23--55.

\bibitem{MR4244837}
{\sc B.~K. Das and H.~Sau}, {\em Algebraic properties of {T}oeplitz operators
  on the symmetrized polydisk}, Complex Anal. Oper. Theory, 15 (2021),
  pp.~Paper No. 60, 28.

\bibitem{DJ77}
{\sc A.~M. Davie and N.~P. Jewell}, {\em Toeplitz operators in several complex
  variables}, J. Functional Analysis, 26 (1977), pp.~356--368.

\bibitem{DSZ12}
{\sc X.~Ding, S.~Sun, and D.~Zheng}, {\em Commuting {T}oeplitz operators on the
  bidisk}, J. Funct. Anal., 263 (2012), pp.~3333--3357.

\bibitem{DS91}
{\sc G.~Dini and A.~Selvaggi~Primicerio}, {\em Proper holomorphic mappings
  between generalized pseudoellipsoids}, Ann. Mat. Pura Appl. (4), 158 (1991),
  p.~219–229.

\bibitem{Douglas98}
{\sc R.~G. Douglas}, {\em Banach algebra techniques in operator theory},
  vol.~179 of Graduate Texts in Mathematics, Springer-Verlag, New York,
  second~ed., 1998.

\bibitem{DP}
{\sc R.~G. Douglas and V.~I. Paulsen}, {\em Hilbert modules over function
  algebras}, vol.~217 of Pitman Research Notes in Mathematics Series, Longman
  Scientific \& Technical, Harlow; copublished in the United States with John
  Wiley \& Sons, Inc., New York, 1989.

\bibitem{EZ}
{\sc A.~Edigarian and W.~o. Zwonek}, {\em Geometry of the symmetrized
  polydisc}, Arch. Math. (Basel), 84 (2005), pp.~364--374.

\bibitem{FK90}
{\sc J.~Faraut and A.~Kor\'{a}nyi}, {\em Function spaces and reproducing
  kernels on bounded symmetric domains}, J. Funct. Anal., 88 (1990),
  pp.~64--89.

\bibitem{Gamelin1969}
{\sc T.~W. Gamelin}, {\em Uniform algebras}, Prentice-Hall, Inc., Englewood
  Cliffs, NJ, 1969.

\bibitem{GG23}
{\sc A.~Ghosh and G.~Ghosh}, {\em {$L^p$} regularity of {S}zeg\"{o} projections
  on quotient domains}, New York J. Math., 29 (2023), pp.~911--930.

\bibitem{kag}
{\sc G.~Ghosh}, {\em The weighted {B}ergman spaces and complex reflection
  groups}, J. Math. Anal. Appl., 548 (2025), pp.~Paper No. 129366, 24.

\bibitem{GN23}
{\sc G.~Ghosh and E.~K. Narayanan}, {\em Toeplitz operators on the weighted
  {B}ergman spaces of quotient domains}, Bull. Sci. Math., 188 (2023),
  pp.~Paper No. 103340, 29.

\bibitem{GZ2023}
{\sc G.~Ghosh and W.~Zwonek}, {\em $2$-proper holomorphic images of classical
  cartan domains}, Indiana Univ. Math. J.,  (2023).

\bibitem{Gottschling1969}
{\sc E.~Gottschling}, {\em Reflections in bounded symmetric domains}, Comm.
  Pure Appl. Math., 22 (1969), pp.~693--714.

\bibitem{GZ97}
{\sc C.~Gu and D.~Zheng}, {\em The semi-commutator of {T}oeplitz operators on
  the bidisc}, J. Operator Theory, 38 (1997), pp.~173--193.

\bibitem{Hahn-Mitchell:1969}
{\sc K.~T. Hahn and J.~Mitchell}, {\em {$H\sp{p}$} spaces on bounded symmetric
  domains}, Trans. Amer. Math. Soc., 146 (1969), pp.~521--531.

\bibitem{Hua1963}
{\sc L.~K. Hua}, {\em Harmonic analysis of functions of several complex
  variables in the classical domains}, American Mathematical Society,
  Providence, RI, 1963.
\newblock Translated from the Russian by Leo Ebner and Adam Kor\'{a}nyi.

\bibitem{Koranyi1965}
{\sc A.~Kor\'{a}nyi}, {\em The {P}oisson integral for generalized half-planes
  and bounded symmetric domains}, Ann. of Math. (2), 82 (1965), pp.~332--350.

\bibitem{KZ2013}
{\sc L.~Kosi\'{n}ski and W.~Zwonek}, {\em Proper holomorphic mappings vs. peak
  points and {S}hilov boundary}, Ann. Polon. Math., 107 (2013), pp.~97--108.

\bibitem{MR2553682}
{\sc Y.~Kosmann-Schwarzbach}, {\em Groups and symmetries}, Universitext,
  Springer, New York, 2010.
\newblock From finite groups to Lie groups, Translated from the 2006 French 2nd
  edition by Stephanie Frank Singer.

\bibitem{Lee10}
{\sc Y.~J. Lee}, {\em Commuting {T}oeplitz operators on the {H}ardy space of
  the polydisk}, Proc. Amer. Math. Soc., 138 (2010), pp.~189--197.

\bibitem{LT09}
{\sc G.~I. Lehrer and D.~E. Taylor}, {\em Unitary reflection groups}, vol.~20
  of Australian Mathematical Society Lecture Series, Cambridge University
  Press, Cambridge, 2009.

\bibitem{MSS18}
{\sc A.~Maji, J.~Sarkar, and S.~Sarkar}, {\em Toeplitz and asymptotic
  {T}oeplitz operators on {$H^2(\Bbb D^n)$}}, Bull. Sci. Math., 146 (2018),
  pp.~33--49.

\bibitem{Mes88}
{\sc M.~Meschiari}, {\em Proper holomorphic maps on an irreducible bounded
  symmetric domain of classical type}, Rend. Circ. Mat. Palermo (2), 37 (1988),
  pp.~18--34.

\bibitem{Misra-SSR-Zhang}
{\sc G.~Misra, S.~Shyam~Roy, and G.~Zhang}, {\em Reproducing kernel for a class
  of weighted {B}ergman spaces on the symmetrized polydisc}, Proc. Amer. Math.
  Soc., 141 (2013), p.~2361–2370.

\bibitem{Pan}
{\sc D.~I. Panyushev}, {\em Lectures on representations of finite groups and
  invariant theory}, https://users.mccme.ru/panyush/notes.html,  (2006).

\bibitem{Rud1980}
{\sc W.~Rudin}, {\em Function theory in the unit ball of {${\bf C}\sp{n}$}},
  vol.~241 of Grundlehren der Mathematischen Wissenschaften, Springer-Verlag,
  New York-Berlin, 1980.

\bibitem{Rud1982}
\leavevmode\vrule height 2pt depth -1.6pt width 23pt, {\em Proper holomorphic
  maps and finite reflection groups}, Indiana Univ. Math. J., 31 (1982),
  p.~701–720.

\bibitem{ST1954}
{\sc G.~C. Shephard and J.~A. Todd}, {\em Finite unitary reflection groups},
  Canad. J. Math., 6 (1954), pp.~274--304.

\bibitem{MR460484}
{\sc R.~P. Stanley}, {\em Relative invariants of finite groups generated by
  pseudoreflections}, J. Algebra, 49 (1977), p.~134–148.

\bibitem{MR117285}
{\sc R.~Steinberg}, {\em Invariants of finite reflection groups}, Canadian J.
  Math., 12 (1960), p.~616–618.

\bibitem{Try13}
{\sc M.~Trybula}, {\em Proper holomorphic mappings, {B}ell's formula, and the
  {L}u {Q}i-{K}eng problem on the tetrablock}, Arch. Math. (Basel), 101 (2013),
  p.~549–558.

\bibitem{Up96}
{\sc H.~Upmeier}, {\em Toeplitz operators and index theory in several complex
  variables}, vol.~81 of Operator Theory: Advances and Applications,
  Birkh\"{a}user Verlag, Basel, 1996.

\bibitem{Zheng98}
{\sc D.~Zheng}, {\em Commuting {T}oeplitz operators with pluriharmonic
  symbols}, Trans. Amer. Math. Soc., 350 (1998), pp.~1595--1618.

\end{thebibliography}

\end{document}